\documentclass[a4paper,10pt,leqno]{amsart}
\usepackage{hyperref}
\usepackage{enumerate,amssymb}
\usepackage[arrow,curve,matrix,tips,2cell]{xy}
\usepackage{tikz-cd}
  \SelectTips{eu}{10} \UseTips
  \UseAllTwocells
\usepackage{pdfsync}

\DeclareMathAlphabet\EuR{U}{eur}{m}{n}
\SetMathAlphabet\EuR{bold}{U}{eur}{b}{n}

\DeclareMathOperator{\ab}{ab}

\DeclareMathOperator{\Aut}{Aut}

\DeclareMathOperator*{\colim}{colim}
\DeclareMathOperator{\cok}{cok}

\DeclareMathOperator{\Ext}{Ext}
\DeclareMathOperator{\geo}{geo}

\DeclareMathOperator{\Hom}{Hom}
\DeclareMathOperator{\id}{id}

\DeclareMathOperator{\im}{im}
\DeclareMathOperator{\ind}{ind}

\DeclareMathOperator{\inc}{inc}

\DeclareMathOperator{\map}{map}

\DeclareMathOperator{\pr}{pr}
\DeclareMathOperator{\per}{per}
\DeclareMathOperator{\res}{res}
\DeclareMathOperator{\rk}{rk}
\DeclareMathOperator{\sAut}{sAut}
\DeclareMathOperator{\sign}{sign}

\DeclareMathOperator{\tr}{tr}

\DeclareMathOperator{\Wh}{Wh}
\DeclareMathOperator{\wh}{Wh}

\newcommand{\strp}{\cals^{\per,s}}
\newcommand{\strone}{\cals^{\langle 1 \rangle,s}}
\newcommand{\strg}{\cals^{\operatorname{geo},s}}

  \newcommand{\IQ}{\mathbb{Q}}
  \newcommand{\IR}{\mathbb{R}}

  \newcommand{\IZ}{\mathbb{Z}}
      \newcommand{\C}{\mathbb{C}}
    \newcommand{\Q}{\mathbb{Q}}
  \newcommand{\R}{\mathbb{R}}
  \newcommand{\Z}{\mathbb{Z}}

  \newcommand{\calf}{\mathcal{F}}
  
  \newcommand{\calh}{\mathcal{H}}

  \newcommand{\calm}{\mathcal{M}}
  \newcommand{\caln}{\mathcal{N}}
  
  \newcommand{\calp}{\mathcal{P}}

  \newcommand{\cals}{\mathcal{S}}

  \newcommand{\bfE}{{\mathbf E}}

  \newcommand{\bfi}{{\mathbf i}}

  \newcommand{\bfK}{{\mathbf K}}
    
   \newcommand{\bfKO}{{\mathbf K}{\mathbf O}}  
  \newcommand{\bfL}{{\mathbf L}}

  \newcommand{\bfp}{{\mathbf p}}
  \newcommand{\bfpr}{{\mathbf {pr}}}

  \newcommand{\wK}{\widetilde K}
  \newcommand{\wL}{\widetilde L}

\newcommand{\curs}{\EuR}
\newcommand{\AB}{\curs{ABELIAN~GROUPS}}

\newcommand{\GROUPOIDS}{\curs{GROUPOIDS}}
\newcommand{\GROUPS}{\curs{GROUPS}}

\newcommand{\Or}{\curs{Or}}

\newcommand{\SPECTRA}{\curs{SPECTRA}}

\newcommand{\eub}[1]{\underline{E}#1}
\newcommand{\edub}[1]{\underline{\underline{E}}#1}
\newcommand{\bub}[1]{\underline{B}#1}

\newcommand{\ol}[1]{\overline{#1}}
\newcommand{\ul}[1]{\underline{#1}}
\newcommand{\wt}[1]{\widetilde{#1}}

\newcommand{\g}{\Gamma}

\DeclareMathOperator{\pt}{\bullet}

\newcommand{\xycomsquare}[8]                      
{\xymatrix{#1 \ar[r]^{#2} \ar[d]^{#4} &
#3 \ar[d]^{#5}  \\
#6\ar[r]^{#7} &
#8
}
}

\newcommand{\xycomsquareminus}[8]                      
{\xymatrix{#1 \ar[r]^-{#2} \ar[d]^-{#4} &
#3 \ar[d]^-{#5}  \\
#6\ar[r]^-{#7} &
#8
}
}

\theoremstyle{plain}
\newtheorem{theorem}{Theorem}[section]

\newtheorem{lemma}[theorem]{Lemma}

\theoremstyle{definition}
\newtheorem{definition}[theorem]{Definition}

\newtheorem{example}[theorem]{Example}

\newtheorem{remark}[theorem]{Remark}

\theoremstyle{remark}

\makeatletter\let\c@equation=\c@theorem\makeatother

\newcommand{\version}[1]              
{\begin{center} last edited on #1\\
last compiled on \today\\
name of tex-file: \jobname
\end{center}
}

\newcounter{commentcounter}

\title[Torus bundles over lens spaces]{Manifolds homotopy equivalent to certain torus bundles over lens spaces}
              \author{James F. Davis}
              \email{jfdavis@indiana.edu}
              \urladdr{http://www.indiana.edu/~jfdavis/}
              \address{Department of Mathematics\\
              Indiana University\\
              Rawles Hall\\
              831 East 3rd St\\
              Bloomington, IN 47405\\
              U.S.A.}
              \author{Wolfgang L\"uck}
        \address{Mathematicians Institut der Universit\"at Bonn\\
                Endenicher Allee 60\\
                53115 Bonn, Germany}
         \email{wolfgang.lueck@him.uni-bonn.de}
          \urladdr{http://www.him.uni-bonn.de/lueck}
              
              \date{July 2019}


\begin{document}

\begin{abstract}
  We compute the topological simple structure set of closed manifolds which occur as total
  spaces of flat bundles over lens spaces $S^l/(\IZ/p)$ with fiber $T^n$ for an odd prime
  $p$ and $l \ge 3$ provided that the induced $\IZ/p$-action on $\pi_1(T^n) = \IZ^n$ is
  free outside the origin. To the best of our know\-ledge this is the first computation of
  the structure set of a topological manifold whose fundamental group is not obtained from
  torsionfree and finite groups using amalgamated and HNN-extensions.  We give a
  collection of classical surgery invariants such as splitting obstructions and
  $\rho$-invariants which decide whether a simple homotopy equivalence from a closed topological
  manifold to $M$ is homotopic to a homeomorphism.
\end{abstract}

\maketitle


\typeout{------------------- Introduction -----------------}
\setcounter{section}{-1}
\section{Introduction}
\label{sec:introduction}


\subsection{Flat torus bundles over lens spaces}
\label{subsec:Flat_torus_bundles_over_lens_spaces}

Throughout this paper we will consider the following setup and notation:

\begin{itemize}

\item Let $p$ be an odd prime; 

\item Let $\rho \colon \IZ/p \to \Aut (\IZ^n) =\mathrm{GL}_n(\IZ)$ be a
group homomorphism so that the induced action of $\Z/p$ on
$\Z^n - \{0\}$ is free;

\item  The homomorphism $\rho$ defines an action of $\Z/p$ on the torus $T^n = \R^n/\Z^n$.  If we
want to emphasize the $\Z/p$-action, we write $T^n_\rho = B\Z^n_\rho$; 

\item Fix a free action of $\Z/p$ on a sphere $S^l$ for an odd integer $l \geq 3$.  We
refer to the orbit space $L^l := S^l/(\Z/p)$ as a \emph{lens space};

\item Define a closed $(n+l)$-manifold
\[
M := T^n_{\rho} \times_{\Z/p} S^l;
\]

\item The fundamental group of $M^{n+l}$ is the semi-direct product denoted by
\[
\Gamma := \Z^n \times_{\rho} \Z/p.
\]
\end{itemize}

We computed the $K$-theory of the $C^*$-algebra of $\g$ in~\cite{Davis-Lueck(2013)}.

An example of such an action $\rho$ is given by the regular
representation $\Z[\Z/p]$ modulo the ideal generated by the norm element
in which case we have $\rho\colon  \Z/p\to \Aut(\Z^{p-1})$.

The action of $\Z/p$ on $\Z^n - \{0\}$ is free if and only if the fixed point set $(T^n)^{\Z/p}$ is finite.
Equip the torus and the sphere with the standard orientations; this determines an orientation on $M$.
Since the $\Z/p$-action on $S^l$ is free, there is a fiber bundle $T^n_{\rho} \to M^{n+l} \to L^l$.
It is worth noting that $T^n$, $L^{\infty}$ and
$T^n_{\rho} \times_{\Z/p} S^\infty$ are models for $B\IZ^n$,  $B\Z/p$, and $B\Gamma$ respectively.
Notice that our assumptions imply that $\dim(M) = n+l \ge 5$.

  Next we summarize our main results.


\subsection{The geometric topological simple structure set of $M$}
\label{subsec:The_geometric_topological_simple_structure_set_of_M}

We will show in Theorem~\ref{the:The_geometric_simple_structure_set_of_M}~%
\ref{the:The_geometric_simple_structure_set_of_M:comp}

\begin{theorem}[The geometric topological simple structure set of $M$] 
As an abelian group we have
\begin{eqnarray*}
\strg(M) &\cong & 
\IZ^{p^k(p-1)/2} \oplus \bigoplus_{i = 0}^{n-1} L_{n-i}(\IZ)^{r_i},
\end{eqnarray*}
where the natural number $k$ is determined by the equality $n = k(p-1)$ and
the numbers $r_i$ are defined in~\eqref{r-numbers}.
\end{theorem}

In particular, the structure set is infinite.

To our knowledge this is the first computation of the structure set of a topological
manifold whose fundamental group is not obtained from torsionfree and finite groups using
amalgamated and HNN-extensions.  The computation is rather involved and based on the
Farrell-Jones Conjecture.  We also compute the periodic structure sets of $B\Gamma$ and of
$M$ and prove detection results for these structure sets. The notion of a structure set is recalled
in Section~\ref{subsec:structure_sets}.  Its study is motivated by the question of determining the
homeomorphism classes of closed manifolds homotopy equivalent to $M$.


\subsection{Homotopy equivalence versus homeomorphism}
\label{subsec:Homotopy_equivalence_versus_homeomorphism}

The following result gives a criterion when a simple homotopy equivalence of closed
manifolds with $M$ as target is homotopic to a homeomorphism. It is proved in
Section~\ref{subsec:Invariance_for_Detecting_the_structure_set_of_M}.

\begin{theorem}[Simple homotopy equivalence versus homeomorphism]%
\label{the:zero_in_structure_set_intro}
  Let $h \colon N \to M$ be a simple homotopy equivalence with a closed topological
  manifold $N$ as source and the manifold $M$
  of~Subsection~\ref{subsec:Flat_torus_bundles_over_lens_spaces} as target. Then $h$ is
  homotopic to a homeomorphism if and only only if the following conditions are satisfied:

\begin{itemize}

\item\label{the:zero_in_structure_set_intro:splitting} 
Vanishing of splitting obstructions:\\[1mm]
Let $\overline{h} \colon \overline{N} \to T^n \times S^l$ 
be obtained from $h$ by pulling back
 the $\IZ/p$-covering $T^n \times S^l \to M$. 
Consider any nonempty subset $J \subset \{1,2, \ldots, n\}$.
Let $T^J \times \pt\subset T^n \times S^l$ 
be the obvious $|J|$-dimensional submanifold. By making
$\overline{h}$ transversal to $T^J \times \pt$, we obtain a normal
map $(\overline{h})^{ -1}(T^J \times \pt) \to T^J \times \pt$ which defines a surgery
obstruction in $L_{|J|}(\IZ)$. 

This obstruction has to be zero;

\item\label{the:zero_in_structure_set_intro:Rho}
  Equality of $\rho$-invariants:\\[1mm]
  Consider any subgroup $P \subset \Gamma$ of order $p$.  Let $P'$ be the image of $P$
  under the abelianization map $\pr \colon \Gamma \to \Gamma_{\ab}$.  Let $M_P \to M$ be
  the cover corresponding to the subgroup $\pr^{-1}(P')$.  Let $h_P : N_P \to M_P$ be the
  corresponding covering simple homotopy equivalence.

Then we must have the equality of $\rho$-invariants in $ \wt R(P')^{(-1)^{(n+l+1)/2}}[1/p]$
\[
\rho(N_P \to BP') = \rho(M_P \to BP').
\]
\end{itemize}
\end{theorem}


\subsection{Acknowledgments}  
\label{subsec:Acknowledgements}

The first author was supported by NSF grant DMS 1615056.
The paper has been supported financially by  the ERC Advanced Grant ``KL2MG-interactions''
(no.  662400) of the second author granted by the European Research Council and
by the Deutsche Forschungsgemeinschaft (DFG, German Research Foundation)
under Germany's Excellence Strategy \--- GZ 2047/1, Projekt-ID 390685813.  We wish to thank the referee for careful reading and helpful suggestions.

The paper is organized as follows:
\tableofcontents

Here is a detailed outline of the paper.  Using the Farrell-Jones Conjecture, it is
straightforward to compute $K_1$ and $K_0$ of the group ring $\Z \g$ (see
Section~\ref{sec:K_theory_of_group_ring}).  Computing the algebraic $L$-theory
(Section~\ref{sec:L_theory_of_group_ring}) is more difficult, we use homological
computations from our previous paper~\cite{Davis-Lueck(2013)}, and a result of Land and
Nikolaus~\cite{Land-Nikolaus(2018)} which generalizes a result of
Sullivan~\cite{Sullivan(1966PhD), Sullivan(1966notes)} comparing $L$-theory spectra with
topological $K$-theory spectra after inverting 2.  As with much of the rest of the paper,
we compute the algebraic $L$-theory at $p$ and away from $p$.

Our goal is to compute and detect the geometric structure set of $M$.  However, it is much
easier to compute the periodic structure set (also known as the algebraic structure set)
of the classifying space $B\g$.  Indeed, as a simple application of the Farrell-Jones
Conjecture, we show in Section~\ref{sec:The_periodic_simple_structure_set_of_BGamma}
\[
  \bigoplus_{P \in (\calp)} \strp_m(BP) \xrightarrow{\cong} \strp_m(B\g).
\]

In Section~\ref{subsec:The_periodic_simple_structure_set_of_BG_finite_p-group} we use
equivariant $KO$-homology to show for any odd order $p$-group $G$, that $\strp_m(BG)$ is a
finitely generated free $\Z[1/p]$-module.
Section~\ref{sec:The_periodic_simple_structure_set_of_M} is the heart of the paper, where
we compute the periodic structure set of $M$, working at $p$ and away from $p$.  In
Section~\ref{sec:The_geometric_simple_structure_set_of_M} we give the computation of the
geometric structure set of $M$, as well as detection by algebraic topological invariants.
In Section~\ref{subsec:Invariance_for_Detecting_the_structure_set_of_M}, we detect the
structure set by the geometric invariants given in
Theorem~\ref{the:zero_in_structure_set_intro}: splitting obstructions and the
$\rho$-invariant.  Finally, in Section~\ref{sec:Appendix:Open_questions} we mention some
basic questions which we did not answer, in the hope that these questions are accessible
and will stimulate future work.


\typeout{----------- Preliminaries about the group $\IZ^n \rtimes \IZ/p$ ---------------}

\section{Preliminaries about the group $\IZ^n \rtimes \IZ/p$}
\label{subsec:Preliminiaries_about_the_group_Zn_rtimes_Z/p}

In this section we collect various facts about $\Gamma$ from~\cite[Lemma~1.9]{Davis-Lueck(2013)}.

\begin{lemma}\label{lem:preliminaries_about_Gamma_and_Zn_rho}
  \begin{enumerate}
  \item\label{lem:preliminaries_about_Gamma_and_Zn_rho:ideals} Let
    $\zeta = e^{2 \pi i/p}$.  There are nonzero ideals $I_1, \dots , I_k$ of $\Z[\zeta]$
    and isomorphisms of $\Z[\Z/p]$-modules
    \begin{align*}
      \IZ^n & \cong I_1 \oplus \dots \oplus I_k; \\
      \IZ^n \otimes \IQ & \cong  \IQ(\zeta)^k.
    \end{align*}
    Hence $n = k(p-1)$;

  \item\label{lem:preliminaries_about_Gamma_and_Zn_rho:finite_subgroups} Each nontrivial
    finite subgroup $P$ of $\Gamma$ is isomorphic to $\IZ/p$ and its Weyl group
    $W_{\Gamma}\!P := N_{\Gamma}P/P$ is trivial;

  \item\label{lem:preliminaries_about_Gamma_and_Zn_rho:list_of_finite_subgroups} There
    are isomorphisms
\[
H^1(\IZ/p;\IZ^n) \xrightarrow{\cong} \cok(\rho -\id \colon \IZ^n \to \IZ^n) \cong
(\IZ/p)^k;
\]
and a bijection
\[
\cok\bigl(\rho -\id \colon \IZ^n \to \IZ^n\bigr) \xrightarrow{\cong} \calp := \{(P)\mid P
\subset \Gamma, 1 < |P| < \infty \}.
\]
Here $\calp$ is the set of conjugacy classes $(P)$ of nontrivial subgroups of finite
order.  If we fix an element $s \in \Gamma$ of order $p$, the bijection sends the element
$\overline{u} \in \IZ^n/(1-\rho)\IZ^n$ to the conjugacy class of the subgroup of order $p$ generated by
$us$;

\item\label{lem:preliminaries_about_Gamma_and_Zn_rho:order_of_calp} We have
  $|\calp| = p^k$;

\item\label{lem:preliminaries_about_Gamma_and_Zn_rho:fixed_set} There is a
  bijection from the $\IZ/p$-fixed set of the $\IZ/p$-space $T^n_{\rho} :=
  \R^n_\rho/\Z^n_\rho$ to $H^1(\Z/p; \Z^n_\rho)$.  In particular
  $(T^n_\rho)^{\IZ/p}$ consists of $p^k$ points;

\item\label{lem:preliminaries_about_Gamma_and_Zn_rho:commutator}
  $[\Gamma,\Gamma] = \im\left(\rho - \id \colon \IZ^n \to \IZ^n\right)$;

\item\label{lem:preliminaries_about_Gamma_and_Zn_rho:abelianization}
  $\Gamma/[\Gamma,\Gamma] \cong \cok(\rho -\id \colon \IZ^n \to \IZ^n) \oplus
  \IZ/p = (\IZ/p)^{k+1}$.

\end{enumerate}

\end{lemma}


\typeout{--------- Preliminaries about the Farrell-Jones Conjecture ------}

\section{Preliminaries about the Farrell-Jones Conjecture}
\label{subsec:Preliminaries_about_the_Farrell-Jones_Conjecture}

To classify high-dimensional manifolds one uses the surgery exact sequence.  One of the
terms in the surgery exact sequence is the 4-periodic $L$-group $L_*(\Z G)$, where $G$ is
the fundamental group of the manifold under consideration.  Although the $L$-groups are
algebraically defined, when $G$ is infinite the computation of the $L$-groups is done by a
mix of algebraic, topological, and geometric methods.  This is encoded in the
Farrell-Jones Conjecture, which can be stated in terms of an equivariant
homology theory in the sense of~\cite[Section~1]{Lueck(2002b)} as follows.

Let $\bfE\colon  \GROUPOIDS \to \SPECTRA$ be a covariant functor 
from the category of small groupoids to the category of spectra, which is \emph{homotopy invariant}, i.e., it
sends an equivalence of groupoids to a weak homotopy equivalence of spectra.  Given a cellular
map $X \to Y$ of $G$-$CW$-complexes for a (discrete) group $G$,
Davis-L\"uck~\cite{Davis-Lueck(1998)} define
\[
H^G_m(X \to Y ; \bfE) := \pi_m\left(\map_G(G/-,\text{cone}(X\to Y)) \wedge_{\Or(G)} \bfE(\overline{G/-})\right),
\]
where cone refers to the mapping cone, $\Or(G)$ to the orbit category of $G$, and
$\overline{G/H}$ to the groupoid associated to the $G$-set $G/H$.  This defines a
$G$-homology theory $H^G_*$ on the category of $G$-$CW$-complexes. Its coefficients are
given by $H^G_m(G/H ; \bfE) = \pi_m(\bfE(H))$.

An equivariant homology theory $\calh^{?}_*$ in the sense of~\cite[Section~1]{Lueck(2002b)}
assigns to every discrete group $G$ a $G$-homology theory $\calh^G_*$ on the category of
$G$-$CW$-complexes.  Given a group homomorphism $\alpha\colon  G \to H$, there is a
corresponding map of abelian groups
$\ind_{\alpha} \colon\calh^G_m(X,A) \to \calh^H_m(H \times_{\alpha} (X,A))$.  The axioms for an
equivariant homology theory are satisfied when $\bfE\colon\GROUPOIDS \to \SPECTRA$ is a
homotopy invariant functor and $\calh^G_m(X)$ is defined as above,
see~\cite[Proposition~157 on page~796]{Lueck-Reich(2005)}.

Examples of such $\GROUPOIDS$-spectra are $ \bfK$ and $ \bfL^{\langle -\infty \rangle}$
defined in~\cite[Section 2]{Davis-Lueck(1998)}. Here
$\pi_m(\bfK(\ol{G/H})) = K_m(\Z H)$ and
$\pi_m(\bfL(\ol{G/H})) = L_m^{\langle -\infty \rangle}(\Z H)$.

A \emph{family} $\calf$ of subgroups of $G$ is a collection of subgroups which is nonempty
and closed under conjugation and under taking subgroups.  The \emph{classifying space
  $E_\calf G$ for group actions with isotropy in $\calf$} is characterized up to $G$-homotopy equivalence as a
$G$-CW-complex where $(E_\calf G)^H$ is empty if $H \not \in \calf$ and contractible if
$H \in \calf$.  We write $\eub G$ for the classifying space when $\calf$ is the family of
finite subgroups and $\edub G$ for the classifying space when the $\calf$ is the family of
virtually cyclic subgroups. For more information about these spaces we refer for instance 
to~\cite{Lueck(2005s)}. 

The Farrell-Jones Conjecture  for the group $G$, which was originally stated 
in~\cite[1.6 on page~257]{Farrell-Jones(1993a)}, predicts that for all $m \in \IZ$
the projection $\edub G \to \pt$ induces isomorphisms
\begin{eqnarray*}
H^G_m(\edub G;\bfK)  
& \to &  
H^G_m(\pt;\bfK) = K_m(\IZ G);
\\
H^G_m(\edub G;\bfL^{\langle -\infty \rangle})  
& \to &  
H^G_m(\pt;\bfL^{\langle -\infty \rangle}) = L^{\langle -\infty \rangle}_m(\IZ G).
\end{eqnarray*}

We now specialize to the group $\Gamma = \Z^n \rtimes \Z/p$.  The first point is that the
Farrell-Jones Conjecture in $K$- and $L$-theory holds for $\Gamma$
by~\cite{Bartels-Lueck(2012annals)}. Since the only subgroups which are virtually cyclic
and not finite are infinite cyclic and since the Farrell-Jones Conjecture holds for
infinite cyclic groups, the transitivity principle~\cite[Theorem~A.10]{Farrell-Jones(1993a)} 
or~\cite[Theorem~65 on page~742]{Lueck-Reich(2005)} shows that
\begin{eqnarray*}
H^\Gamma_m(\eub \Gamma; \bfL^{\langle -\infty \rangle})  
& \xrightarrow{\cong} &
H^\Gamma_m(\edub \Gamma; \bfL^{\langle -\infty \rangle});
\\
H^\Gamma_m(\eub \Gamma; \bfK)  
& \xrightarrow{\cong} &
H^\Gamma_m(\edub \Gamma; \bfK),
\end{eqnarray*}
are bijective for all $m \in \IZ$. Hence we get

\begin{theorem}[Farrell-Jones Conjecture for $\Gamma$]\label{the:FJC_L-infty_for_Gamma)}
The projection  $\eub \Gamma \to \pt$ induces for all $m \in \IZ$ bijections
\begin{eqnarray*}
H^{\Gamma}_m(\eub \Gamma;\bfK)  
& \to &  
H^{\Gamma}_m(\pt;\bfK) = K_m(\IZ \Gamma);
\\
H^{\Gamma}_m(\eub \Gamma;\bfL^{\langle -\infty \rangle})  
& \to &  
H^{\Gamma}_m(\pt;\bfL^{\langle -\infty \rangle}) = L^{\langle -\infty \rangle}_m(\IZ \Gamma).
\end{eqnarray*}
\end{theorem}

\begin{remark}
  Note that $\g = \Z^n \rtimes \Z/p$ acts affinely on $\Z^n$ where $\Z^n$ acts by
  translation and $\Z/p$ acts by the map $\rho$.  By extending scalars, $\g$ acts properly and cocompactly on
  $\R^n$.  Hence $\g$ is a crystallographic group and $\R^n$ can be taken as a model for
  $\eub \g$.
\end{remark}


\typeout{------------------------ Algebraic $K$-theory of the group ring  -------------------------}


\section{Algebraic $K$-theory of the group ring} 

\label{sec:K_theory_of_group_ring}

For a group $G$ and an integer $m$, define $\wh_m(G)$ to be the
homotopy groups of the homotopy cofiber of the assembly map in
algebraic $K$-theory, i.e.
\[
\wh_m(G) = H^G_m(EG \to \pt; \bfK) 
\]
where $\bfK = \bfK_{\Z}$ is the algebraic $K$-theory spectrum over the
orbit category of~\cite[Section~2]{Davis-Lueck(1998)} with
$\pi_n(\bfK(\ol{G/H})) = K_n(\Z H)$.  Hence $\wh_1(G)$ is the
classical Whitehead group
$\wh(G) = \cok( \{\pm 1\} \times G^{\ab} \to K_1(\Z G))$, $\wh_0(G)$
agrees with $\wt K_0(\Z G) = \cok(K_0(\Z) \to K_0(\Z G))$, and
$\wh_{n}(G)$ is $K_{n}(\Z G)$ for $n \le -1$.

The space $\eub \Gamma$ can be profitably analyzed using the following  cellular
$\Gamma$-pushout, see~\cite[Corollary~2.11]{Lueck-Weiermann(2012)},
\begin{eqnarray}
  &
  \xycomsquareminus{\coprod_{(P) \in \calp} \Gamma \times _P EP}{}{E\Gamma}
  {}{}
  {\coprod_{(P) \in \calp} \Gamma/P}{}{\eub{\Gamma}}
  &
\label{G-pushoutfor_EGamma_to_eunderbar_Gamma}
\end{eqnarray}

This leads to the following result taken from~\cite[Lemma 7.2~(ii)]{Davis-Lueck(2013)}.

\begin{lemma}\label{lem:MV_sequence} Let $\calh^?_*$ be an equivariant homology theory
  in the sense of~\cite[Section~1]{Lueck(2002b)}.  Then there is a long exact sequence
\begin{multline*}
\dots \to \calh^\g_{m+1}(\eub  \g)  \xrightarrow{\ind_{\g \to 1}} \calh_{m+1}(\bub \g) \to \bigoplus_{P \in (\calp)} \wt \calh^P_m(\pt) 
\\
\xrightarrow{\varphi_m} \calh^\g_m(\eub \g)  \xrightarrow{\ind_{\g \to 1}}  \calh_m(\bub \g) \to \dots ,
\end{multline*}
where $\widetilde{\calh}_m^{P}(\pt)$ is the kernel of the induction map
$\ind_{P \to 1} \colon \calh_m^{P}(\pt) \to \calh_m(\pt)$, the map $\varphi_m$ is induced by the various
inclusions $P \to \Gamma$, and $\bub{\Gamma} := \Gamma \backslash \eub{\Gamma}$.

The map 
\[
\ind_{\Gamma \to 1}[1/p] \colon  \calh^{\Gamma}_m(\eub{\Gamma})[1/p]
\to \calh_m(\bub{\Gamma})[1/p]
\]
is split surjective.
\end{lemma}

\begin{theorem}[Computation of $\wh_m(\Gamma)$]\label{the:algebraic_K-theory}
For every $n \in \Z$, 
\[
\bigoplus_{P \in (\calp)} \wh_n(P) \xrightarrow{\cong}  \wh_n(\Gamma) 
\]
Furthermore, for $p$ an odd prime, $\wh(\Z/p) \cong \Z^{(p-3)/2}$,
$\wt K_0(\Z[\Z/p])$ is the ideal class group $C(\Z[\exp(2\pi i/p)])$
and hence is finite, and  $K_{n}(\Z[\Z/p]) = 0$ for $n \le -1$.
\end{theorem}
\begin{proof}
The isomorphism $\oplus \wh_n(P) \cong \wh_n(\Gamma)$ is a direct consequence of 
Theorem~\ref{the:FJC_L-infty_for_Gamma)} and  the 
$G$-pushout~\eqref{G-pushoutfor_EGamma_to_eunderbar_Gamma}.
See also~\cite[Theorem 1.8]{Lueck-Rosenthal(2014)},~\cite[Theorem 5.1(d)]{Davis-Lueck(2003)} 
and~\cite[Theorem 0.2]{Lueck-Stamm(2000)}.

The computation of $\wh(\Z/p)$ for an odd prime $p$ (and much more
information about the Whitehead group for finite groups) can be found
in~\cite{Oliver(1988)}, a discussion of $\wt K_0(\Z [\Z/p])$
in~\cite{Milnor(1971)}, and the vanishing of $K_{n}(\Z [\Z/p])$ for $n \le -1$
in~\cite{Carter(1980)}.
\end{proof}

Theorem~\ref{the:algebraic_K-theory} is consistent with~\cite[Theorem~1.8~(i)]{Lueck-Rosenthal(2014)}.


\typeout{------------------------ Algebraic $L$-theory of the group ring  -------------------------}


\section{Algebraic $L$-theory of the group ring} 

\label{sec:L_theory_of_group_ring}

In this section we compute the $L$-groups of $\IZ\Gamma$ for all decorations.


\subsection{Decorated $L$-groups}
\label{subsec:Decorated_L-groups}

We first discuss the so-called decorated versions of $L$-theory
$L^{\langle -i \rangle}_*(\Z G)$ for any group $G$ for $ i = 2, 1, 0, -1, -2, \ldots$ and
$i = -\infty$. We recall briefly a few facts; more information can be found
in~\cite{Ranicki(1992a)}.  For $m \in \Z$, these are functors
$L_m^{\langle i \rangle} \colon\GROUPS \to \AB$ which are 4-periodic in the sense that
$L_m^{\langle i \rangle}(\Z G) = L_{m+4}^{\langle i \rangle}(\Z G)$.  There are natural maps
\[
L^{\langle i+1 \rangle}_*(\IZ G) \to L^{\langle i \rangle}_*(\IZ G)
\]
and one defines
\[
L^{\langle -\infty \rangle}_*(\IZ G) = \colim_{i \to -\infty} L^{\langle i \rangle}_*(\IZ G).
\]
One sometimes writes
\begin{eqnarray*}
L^s_*(\IZ G) & = & L^{\langle 2 \rangle}_*(\IZ G);
\\
L^h_*(\IZ G) & = & L^{\langle 1 \rangle}_*(\IZ G);
\\
L^p_*(\IZ G) & = & L^{\langle 0 \rangle}_*(\IZ G).
\end{eqnarray*}

The $L^s$-groups are bordism groups of algebraic Poincar\'e complexes with based modules
and simple Poincar\'e duality; they are useful in classifying manifolds.  The $L^h$-groups
are bordism groups of algebraic Poincar\'e complexes with free modules; they are useful
for studying the existence question of when a space has the homotopy type of a manifold.
The $L^p$-groups are bordism groups of algebraic Poincar\'e complexes with projective
modules.  The $L^{\langle -\infty \rangle}$-groups are useful for the Farrell-Jones
Conjecture.

We will often use the Ranicki-Rothenberg exact sequence for a group $G$,
see~\cite[Theorem~7.12 on page~146]{Ranicki(1992a)}

\begin{multline}
\cdots \to L^{\langle i +1 \rangle}_m(\IZ G) \to  L^{\langle i \rangle}_m(\IZ G) 
\to \widehat{H}^m(\IZ/2;\Wh_i(G))
\\
\to L^{\langle i + 1 \rangle}_{m-1}(\IZ G) \to  L^{\langle i \rangle}_{m-1}(\IZ G) 
\to \cdots.
\label{Rothenberg_sequence}
\end{multline}

For any decoration $i$, there is a homotopy invariant functor
\[
\bfL^{\langle i \rangle} \colon\GROUPOIDS \to \SPECTRA
\]
satisfying $\pi_m(\bfL^{\langle i \rangle}(\ol{G/H})) = L_m^{\langle i \rangle}(\Z H)$.  
Farrell and Jones~\cite{Farrell-Jones(1993a)}  conjecture that
\[
H_*^G(\edub G; \bfL^{\langle -\infty \rangle}) \xrightarrow{\cong} H_*^G(\pt; \bfL^{\langle -\infty \rangle})
\]
for all groups $G$.   However, the decorated version of the assembly map 
\[
H_*^G(\edub G; \bfL^{\langle i \rangle}) \xrightarrow{\cong} H_*^G(\pt; \bfL^{\langle i \rangle})
\]
need not be bijective in general, for example, it is not bijective for the group
$G = \Z^2 \times \Z/29$ for the decorations $p$, $h$, and $s$, see~\cite[Example~14]{Farrell-Jones-Lueck(2002)}.  
However, the Farrell-Jones Conjecture for the group $\g$
holds for all $i$ as we show below. This will be important since the $L^s$-version is the geometrically significant one.


\subsection{The $\langle -\infty \rangle$-decoration}
\label{subsec:The_langle-infty_rangle_decoration}

In this subsection we compute $ L^{\langle - \infty \rangle}_m(\Z\Gamma)$ using the
Farrell-Jones Conjecture, see Theorem~\ref{the:FJC_L-infty_for_Gamma)}.  The $L$-theory of
$\g = \Z^n \rtimes \Z/p$ is, in some sense, built from the $L$-theory of $\Z^n$ and
$\Z/p$, so we first review these.

The Farrell-Jones Conjecture in $K$-theory holds for the torsion-free group $\Z^n$.  It
follows that $\wh_m(\Z^n)=0$ for all $m \in \Z$ and for all $n \in \Z_{\geq 0}$.  Thus the
maps
\begin{align}
L_m^{\langle i  \rangle}(\Z[\Z^n]) 
& \xrightarrow{\cong} 
L_m^{\langle -\infty \rangle}(\Z[\Z^n]) 
\label{l(Z)_and_decos}
\end{align}
are bijections for  $i$, $m$, and $n$  and the  map of spectra
\begin{align}
\bfL^{\langle i \rangle}(\Z[\Z^n]) 
& \xrightarrow{\simeq} 
\bfL^{\langle -\infty \rangle}(\Z[\Z^n]) \label{bfl(Z)_and_decos}
\end{align}
is a weak homotopy equivalence for all $i$ and $n$.  Hence we will omit the decoration and
refer to $L_m(\Z[\Z^n])$ and $\bfL(\Z[\Z^n])$.

When $n = 0$, the $L$-groups are well known, essentially  due to Kervaire-Milnor,
\[
L_m(\Z) = 
\begin{cases}
 \Z & m \equiv 0 \pmod 4; \\
  0 & m \equiv 1 \pmod 4; \\
   \Z/2 & m \equiv 2 \pmod 4; \\
 0 & m \equiv 3 \pmod 4, \\
\end{cases}
\]
where the map to $\Z$ is given by the signature divided by 8 and the map to $\Z/2$ is
given by the Arf invariant.  Since the Farrell-Jones Conjecture in $L$-theory holds for
the torsion-free group $\Z^n$,
\begin{equation}\label{L-theory_of_Z_upper_n}
L_m(\Z[\Z^n]) \xleftarrow{\cong} H_m(B\Z^n; \bfL(\Z)) = \bigoplus_{i=0}^n L_{m-i}(\Z)^{\binom{n}{i}}.
\end{equation}

For $p$ an odd prime, we get
from~\cite[Theorem~1]{Bak(1975)},~\cite[Theorem~1,2,  and 3]{Bak(1978)}, and~\cite[Theorem~10.1]{Hambleton-Taylor(2000)}
\begin{equation*}
 L_m^{\langle 0 \rangle}(\Z[\Z/p]) = 
\begin{cases}
 \Z^{(p-1)/2} \oplus L_m(\Z) & m \;\text{even;} \\
 0 & m \; \text{odd.}
\end{cases}
\end{equation*}
Since $\wh_i(\Z/p) = 0$ for $i < 0$ by~\cite{Carter(1980)}, we have
\[
L_m^{\langle 0 \rangle}(\Z[\Z/p]) \xrightarrow{\cong}  L_m^{\langle -1 \rangle}(\Z[\Z/p]) \xrightarrow{\cong}  L_m^{\langle -2 \rangle}(\Z[\Z/p])
\xrightarrow{\cong} \cdots \xrightarrow{\cong} L_m^{\langle -\infty \rangle}(\Z[\Z/p]).
\]
Thus
\begin{equation}\label{LminusinftyofZ/p}
 L_m^{\langle -\infty \rangle}(\Z[\Z/p]) = 
\begin{cases}
 \Z^{(p-1)/2}\oplus  L_m(\Z) & m \;\text{even;} \\
 0 & m \;\text{odd.}
\end{cases}
\end{equation}

\begin{theorem}[$L^{\langle -\infty\rangle}$-theory] \
\label{the:L-infty-theory}
There is a long exact sequence 
\begin{multline*}
\cdots \to H_{m+1}(\bub{\Gamma};\bfL(\Z)) 
\to \bigoplus_{(P) \in \calp} \widetilde{L}^{\langle - \infty \rangle}_m(\IZ P) 
\to L^{\langle - \infty \rangle}_m(\IZ\Gamma)
\\
\xrightarrow{\beta_m} H_{m}(\bub{\Gamma};\bfL(\Z)) 
\to \bigoplus_{(P) \in \calp} \widetilde{L}^{\langle - \infty \rangle}_{m-1}(\IZ P) \to \cdots,
\end{multline*} 
where $\widetilde{L}^{\langle - \infty \rangle}_m(\IZ P)$ is the kernel of the map
$L^{\langle - \infty \rangle}_m(\IZ P) \to L_m(\IZ)$
induced by induction with $P \to 1$.

The map $\beta_m[1/p]$ is a split surjection, and thus there is an isomorphism of $\IZ[1/p]$-modules
\[
L^{\langle - \infty \rangle}_m(\IZ\Gamma)[1/p] \cong 
\biggl(\bigoplus_{(P) \in \calp} \widetilde{L}^{\langle - \infty \rangle}_m(\IZ P)[1/p]\biggr)
\oplus H_{m}(\bub{\Gamma};\bfL(\Z))[1/p].
\]
\end{theorem}
\begin{proof}
This follows directly from Theorem~\ref{the:FJC_L-infty_for_Gamma)} and Lemma~\ref{lem:MV_sequence}.
\end{proof}

Next we want to improve  Theorem~\ref{the:L-infty-theory}
by comparing it with the computation of $KO_*(C^*_r(G))$ of~\cite[Theorem~10.1]{Davis-Lueck(2013)}.

The next result is motivated by Sullivan's thesis~\cite{Sullivan(1966PhD)}, but its proof
requires much more, in particular recent results of Land and Nikolaus.

\begin{theorem}[Comparing $L$-theory and topological $K$-theory]
\label{the:Comparing_L-theory_and_topological_K-theory}
There is a natural transformation of
equivariant homology theories
\[
T^?_*(-) \colon H^?_*(-;\bfL^{\langle - \infty \rangle})[1/2] \to KO^?_*(-)[1/2]
\]
such that for every group $G$, every proper $G$-$CW$-complex $X$ and every $m \in \IZ$
the map
\[
T^G_m(X) \colon H^G_m(X;\bfL^{\langle - \infty \rangle})[1/2] \to KO^G_m(X)[1/2]
\]
is an isomorphism.
\end{theorem}
\begin{proof}
  Notice that  we invert
  $2$ so that the decorations do not matter and we  therefore ignore them.  The key
  ingredient is a rigorous construction of a weak homotopy equivalence of spectra
  $\bfKO(A)[1/2] \to \bfL(A)$ for a real $C^*$-algebra $A$ from its topological
  $K$-theory spectrum to its algebraic $L$-theory spectrum after inverting
  $1/2$, which is natural in
  $A$,~\cite[Theorem~C]{Land-Nikolaus(2018)}.  In particular this applies to the real
  group
  $C^*$-algebra and can be extended from groups to groupoids,
  see~\cite[Section~5.2]{Land-Nikolaus(2018)}.  For any group there is a natural map from
  the integral group ring to the real group
  $C^*$-algebra which yields a natural map between the corresponding
  $L$-theory spectra. This construction also extends to groupoids.  Combing these two
  transformations and using the fact that the one in~\cite[Theorem~C]{Land-Nikolaus(2018)} 
  is a weak homotopy equivalence, yields the
  desired transformation $T^?_*(-)$. Recall that a $G$-$CW$-complex is proper if all isotropy groups are finite.  In order to show that
  $T^G_n(X)$ is an isomorphism for every proper $G$-$CW$-complex
  $X$, it suffices to do this in the case $X = G/H$ for a finite subgroup $H \subseteq  G$,
  see~\cite[Lemma~1.2]{Bartels-Echterhoff-Lueck(2008colim)}.
  Hence one needs to show that the change of coefficients map 
  $L_m(\IZ H)[1/2] \to L_m(\IR H)[1/2]$ is bijective for all
  $m$ which is done in~\cite[Proposition~22.34 on page 253]{Ranicki(1992)}.
\end{proof}

If $G$ is a group and $M$ is a left $\Z G$-module, the \emph{invariants} of $M$ are
$M^G := \Hom_{\Z G}(\Z, M) = H^0(BG; M)$ and the \emph{coinvariants} of $M$ are
$M_G := \Z \otimes_{\Z G} M = H_0(BG; M)$.  If $G$ is finite of order $q$, the norm map
$M_G \to M^G$ sending $1 \otimes x$ to $\sum_{g\in G} gx$ is an isomorphism after
tensoring with $\Z[1/q]$.

If $\calh_*$ is a (non-equivariant) generalized homology theory, and $G$ acts on $X$, then
the homology quotient map factors as
$\calh_*(X) \to \calh_*(X)_G \to \calh_*(G \backslash X)$.  If the quotient map
$X \to G \backslash X$ is a regular $G$-cover and $G$ is finite of order $q$, then there
is a transfer map $ \calh_*(G \backslash X) \to \calh_*(X)^G$ so that the composite
$\calh_*(X)_G \to \calh_*(G \backslash X) \to \calh_*(X)^G$ is an isomorphism after
tensoring with $\Z[1/q]$.

\begin{lemma}\label{lem:calh(bub(Gamma))_to_calh(Zn)Z/p}
  Let $\calh_*$ be a any generalized homology theory taking values in $\Z[1/p]$-modules.
  For all $m \in \Z$, the following maps are isomorphisms:
\begin{align*}
\calh_m(B \Z^n_\rho)_{\Z/p} &  \xrightarrow{\cong}  \calh_m(B \Gamma)   \xrightarrow{\cong}  \calh_m(B \Z^n_\rho)^{\Z/p};\\
\calh_m(B \Gamma) & \xrightarrow{\cong}  \calh_m(\bub \Gamma).
\end{align*}
\end{lemma}

\begin{proof} 
Consider the diagram
$$\xymatrix{
  \calh_*(B\Z^n_\rho)_{\Z/p} \ar[r]^\alpha &  \calh_*(B\Gamma) \ar[r]^\beta \ar[d]^\gamma & \calh_*(B\Z^n_\rho)^{\Z/p} \\
  & \calh_*(\bub\Gamma) & }$$ where $\alpha$ and $\gamma$ are the functorial
$\calh_*$-maps and $\beta$ is the transfer.  We will show that $\alpha$, $\beta$, and
$\gamma$ are isomorphisms.

Given a  $\IZ/p$-$CW$-complex $X$, the map
\[
j_m \colon \calh_m(X)_{\Z/p} 
\to  
\calh_m\bigl((\IZ/p)\backslash X\bigr)
\]
is natural in $X$.  Since the functor sending a $\IZ[1/p][\IZ/p]$-module $M$ to $M_{\Z/p}$
is an exact functor, the assignments sending a $\IZ/p$-$CW$-complex to $\calh_m(X)_{\Z/p}$
and to $\calh_m\bigl((\IZ/p)\backslash X\bigr)$ are $\IZ/p$-homology theories and $j_*$ is
a transformation of $\IZ/p$-homology theories.  One easily checks that $j_m$ is a
bijection if $X$ is $(\IZ/p)/H$ for any subgroup $H \subseteq \IZ/p$. It follows that
$j_m$ is a bijection for any $\IZ/p$-$CW$-complex.  Taking $X$ to be
$\Z^n\backslash E \Gamma$, we see $\alpha$ is an isomorphism and taking $X$ to be
$\Z^n\backslash \eub \Gamma$ we see that $\gamma \circ \alpha$ is an isomorphism.  We
commented above that $\beta \circ \alpha$ is an isomorphism.  It follows that all the maps
are isomorphisms.
\end{proof}

We also need some numbers defined in our previous paper~\cite{Davis-Lueck(2013)}.  For
$j,k \in \Z_{\geq 0}$, and for $p$ an odd prime, define
\begin{equation}\label{r-numbers}
  r_{j} := \rk (\Lambda^j(\Z[\zeta_p]^{k})^{\Z/p}),
\end{equation} 
where $\Lambda^j$ means the $j$-th exterior power of a $\Z$-module and $\zeta_p$ is a
primitive $p$-th root of unity.  Thus $r_j = \rk H^j(T^n_\rho)^{\Z/p}$.  When $k = 1$ we
worked out these numbers in~\cite[Lemma~1.22]{Davis-Lueck(2013)}: $r_{j} = 0$ for
$j \geq p$, and for $0 \leq j \leq (p-1)$,
\begin{align*}
  r_j = \frac{1}{p} \left(\binom{p-1}{j} +(-1)^j  (p-1)\right).
\end{align*}

\begin{theorem}[Computation of $L^{\langle -\infty\rangle}(\IZ\Gamma)$] 
\label{the:L-infty(ZGamma)}\
\begin{enumerate}

\item\label{the:L-infty(ZGamma):L2m_plus_1_abstract}
There is an isomorphism
\[
L^{\langle -\infty\rangle}_{2m+1}(\IZ \Gamma) 
\cong
H_{2m+1}(\bub{\Gamma};\bfL(\IZ )).
\]
Transferring to the finite index subgroup $
\Z^n$ of  $\Gamma$ 
induces an isomorphism
\[
L^{\langle -\infty\rangle}_{2m+1}(\IZ \Gamma) 
\xrightarrow{\cong} 
L_{2m+1}(\IZ [\IZ^n_\rho])^{\IZ/p};
\]

\item\label{the:L-infty(ZGamma):L2m_abstract}
There is an exact sequence
\[
0 \to \bigoplus_{(P) \in \calp} \widetilde{L}^{\langle -\infty\rangle}_{2m}(\IZ P) \to 
L^{\langle -\infty\rangle}_{2m}(\IZ \Gamma)  \to 
H_{2m}(\bub{\Gamma};\bfL(\Z )) \to 0,
\]
which splits after inverting $p$;

\item\label{the:L-infty(ZGamma):explicite}
We have
\[
L^{\langle -\infty\rangle}_{m}(\IZ \Gamma)  \cong
\begin{cases}
\IZ^{p^k(p-1)/2} \oplus 
\left(\bigoplus_{i=0}^n  L_{m-i}(\IZ)^{r_{i}}\right) 
& m \; \text{even};
\\
\bigoplus_{i=0}^n  L_{m-i}(\IZ)^{r_{i}}  & m \; \text{odd}.
\end{cases}
\]
\end{enumerate}
\end{theorem}
\begin{proof}~%
\ref{the:L-infty(ZGamma):L2m_plus_1_abstract} and~\ref{the:L-infty(ZGamma):L2m_abstract}
The computation of $L_*^{\langle -\infty \rangle}(\Z[\Z/p])$
in~\eqref{LminusinftyofZ/p} and Theorem~\ref{the:L-infty-theory} implies
that for every $m \in \IZ$ we obtain an isomorphism
\[
L^{\langle -\infty\rangle}_{2m+1}(\IZ \Gamma) 
\xrightarrow{\cong} 
H_{2m+1}(\bub{\Gamma};\bfL(\Z ))
\]
and a short exact sequence
\[
0 \to \bigoplus_{(P) \in \calp} \widetilde{L}^{\langle -\infty\rangle}_{2m}(\IZ P) \to 
L^{\langle -\infty\rangle}_{2m}(\IZ \Gamma)  \to 
H_{2m}(\bub{\Gamma};\bfL(\Z )) \to 0
\]
which splits after inverting $p$.   

It remains to show that transferring to subgroup $
\Z^n$ of  $\Gamma$ 
induces an isomorphism
\[
L^{\langle -\infty\rangle}_{2m+1}(\IZ \Gamma) 
\xrightarrow{\cong} 
L_{2m+1}(\IZ [\IZ^n])^{\IZ/p}.
\]
It suffices to do this after localizing at $p$ and after inverting $p$.   

We first invert 2.
Because of Theorem~\ref{the:Comparing_L-theory_and_topological_K-theory} 
and the fact that the assembly maps
\begin{eqnarray*}
H_m^{\Gamma}(\eub{\Gamma};\bfL^{\langle -\infty \rangle})[1/2] 
& \xrightarrow{\cong} &
H_m^{\Gamma}(\pt;\bfL^{\langle -\infty \rangle})[1/2]  
= L_m^{\langle -\infty \rangle}(\IZ \Gamma)[1/2];
\\
KO_m^{\Gamma}(\eub{\Gamma})[1/2]
& \xrightarrow{\cong} &
KO_m(C^*_r(\Gamma;\IR)),
\end{eqnarray*}
are bijections, see Theorem~\ref{the:FJC_L-infty_for_Gamma)} and~\cite{Higson-Kasparov(2001)},
we obtain a commutative diagram of $\IZ[1/2]$-modules
\[
\xycomsquareminus
{L^{\langle -\infty\rangle}_{2m+1}(\IZ \Gamma)[1/2]}
{\iota^*[1/2]}
{L_{2m+1}(\IZ[\IZ^n])^{\IZ/p}[1/2]}
{\cong}{\cong}
{KO_{2m+1}(C^*_r(\Gamma;\IR))[1/2]}
{\iota^*[1/2]}
{KO_{2m+1}(C^*_r(\IZ^n;\IR))^{\IZ/p}[1/2]}
\]
with bijective vertical arrows. The lower horizontal  arrow
is bijective by Theorem~\cite[Section~11.2]{Davis-Lueck(2013)}.
Hence
\[
\iota^*_{(p)} \colon L^{\langle -\infty\rangle}_{2m+1}(\IZ \Gamma)_{(p)} 
\xrightarrow{\cong}
L_{2m+1}(\IZ[\IZ^n])^{\IZ/p}_{(p)} 
\]
is bijective for every $m \in \IZ$.

Next consider the commutative diagram
\begin{equation} \label{transfer_lemma}
\xycomsquareminus
{H_{2m+1}(B\g;\bfL(\Z))}
{\iota^*}
{H_{2m+1}(B\Z^n;\bfL(\Z))^{\Z/p}}
{A_\g}{A_{\Z^n}^{\Z/p}}
{L^{\langle -\infty\rangle}_{2m+1}(\IZ \Gamma)}
{\iota^*}
{L_{2m+1}(\IZ [\Z^n])^{\Z/p}.}
\end{equation}
We will show that all maps are isomorphisms after inverting $p$.
Lemma~\ref{lem:calh(bub(Gamma))_to_calh(Zn)Z/p} shows this is true for the top horizontal map.
The Farrell-Jones Conjecture in $L$-theory for $\Z^n$ (see~\eqref{L-theory_of_Z_upper_n}) shows
this holds for the right vertical map.  To show this holds for $A_\g$, first rewrite it as
$A_\g : H_{2m+1}(B\g;\bfL(\Z)) = H_{2m+1}^\g(E\g;\bfL^{\langle - \infty \rangle}) \to
H_{2m+1}^\g(\eub\g;\bfL^{\langle - \infty \rangle})$, using the Farrell-Jones Conjecture
in $L$-theory for $\g$.  The map
$\ind_{\g \to 1} : H_{2m+1}^\g(\eub\g;\bfL^{\langle - \infty \rangle}) \to
H_{2m+1}(\bub\g;\bfL(\Z))$ is an isomorphism after inverting $p$ by
Lemma~\ref{lem:MV_sequence} and the vanishing of
$L_{2m+1}^{\langle - \infty \rangle}(\Z[\Z/p])$.  The composite
$\ind_{\g \to 1} \circ A_\g$ is an isomorphism after inverting $p$ by
Lemma~\ref{lem:calh(bub(Gamma))_to_calh(Zn)Z/p}.  Thus we can conclude that $A_\g$ is an
isomorphism after inverting $p$.

Thus the bottom row
$\iota^* : L^{\langle -\infty\rangle}_{2m+1}(\IZ \Gamma) \to L_{2m+1}(\IZ [\Z^n])^{\Z/p}$
of~\eqref{transfer_lemma} is also an isomorphism after inverting $p$.  We have shown that
$\iota^*_{(p)}$ and $\iota^*[1/p]$ are isomorphisms, so we conclude that $\iota^*$ is an
isomorphism.  \\[1mm]~%
\ref{the:L-infty(ZGamma):explicite} It suffices to show that the abelian groups in
question are isomorphic after inverting $2$ and after inverting $p$.  We conclude
from~\cite[Theorem~10.1]{Davis-Lueck(2013)} and
Theorem~\ref{the:Comparing_L-theory_and_topological_K-theory} that this is the case after
inverting $2$. It remains to treat the case, where we invert $p$. Because of
assertion~\ref{the:L-infty(ZGamma):L2m_abstract},\eqref{LminusinftyofZ/p}, and
Lemma~\ref{lem:calh(bub(Gamma))_to_calh(Zn)Z/p}, it remains to prove
\begin{eqnarray}  \label{fixed_normal_invariants}
L_{m}(\IZ[\IZ^n])^{\IZ/p}[1/p]
& \cong & 
\bigoplus_{i=0}^n  \left(L_{m-i}(\IZ)\right)^{r_{i}}[1/p].
\label{L(ZZn)Z/p}
\end{eqnarray}

Since $L_{m}(\IZ[\IZ^n]) \cong H_m(B\IZ^n;\bfL(\Z))$,
it remains to show
\begin{eqnarray}
H_m(B\IZ^n;\bfL(\Z))^{\IZ/p}[1/p] 
& \cong &
\bigoplus_{i=0}^n  \left(L_{m-i}(\Z)\right)^{r_{i}}[1/p].
\label{H_m(BZn;Llangle-inftyrangle)Z/p[1/p]}
\end{eqnarray} 
The Atiyah-Hirzebruch spectral sequence converging to 
$H_m(B\IZ^n;\bfL(\Z))$ collapses, since $B\IZ^n$ is $T^n$ and therefore one can compute
$H_m(B\IZ^n;\bfL(\Z))$ directly. Hence we obtain a filtration
of $\IZ[\IZ/p]$-modules 
\[
0 = F_{n+1,m-n-1} \subseteq F_{n,m-n} \subseteq \cdots \subseteq 
F_{1,m-1} \subseteq  F_{0,m} = H_m(B\IZ^n;\bfL(\Z))
\]
together with exact sequences of $\IZ[\IZ/p]$-modules
\[
0 \to F_{i+1,m-i-1} \to F_{i,m-i} 
\xrightarrow{q} H_i(B\IZ^n_{\rho}) \otimes_{\IZ} L^{\langle - \infty \rangle}_{m-i}(\IZ) \to 0
\]
which splits as short exact sequence of $\IZ$-modules. 
Let $s \colon H_i(B\IZ^n_{\rho}) \otimes_{\IZ} L_{m-i}(\IZ)
\to  F_{i,m-i}$ be a $\IZ$-map with $q \circ s = \id$.
Then
\[
\widetilde{s} \colon H_i(B\IZ^n_{\rho}) \otimes_{\IZ} L_{m-i}(\IZ)
\to  F_{i,m-i}, \quad x \mapsto \sum_{g \in \IZ/p} g\cdot s(g^{-1}\cdot x)
\] 
is a $\IZ [\IZ/p]$-map with $q \circ \widetilde{s} = p \cdot \id$.
 
We obtain
an exact sequence of $\IZ[1/p]$-modules
\begin{multline*}
0 \to (F_{i+1,n-i+1})^{\IZ/p}[1/p] \to (F_{i,n-i})^{\IZ/p}[1/p] 
\\
\xrightarrow{q^{\IZ/p}[1/p]} 
H_i(B\IZ^n_{\rho})^{\IZ/p}[1/p] \otimes_{\IZ} L_{m-i}(\IZ) \to 0.
\end{multline*}
Since $q^{\IZ/p}[1/p] \circ \widetilde{s}^{\IZ/p}[1/p]$ is the automorphism
$p \cdot \id$ of 
$H_i(B\IZ^n_{\rho})^{\IZ/p}[1/p] \otimes_{\IZ} L_{m-i}(\IZ)$ 
this short exact sequence  of $\IZ[1/p]$-modules splits and we obtain an
$\IZ[1/p]$-iso\-mor\-phism
\[
(F_{i,m-i})^{\IZ/p}[1/p]  \cong (F_{i+1,m-i-1})^{\IZ/p}[1/p] \oplus 
\left(H_i(B\IZ^n_{\rho})^{\IZ/p}[1/p] 
\otimes_{\IZ} L_{m-i}(\IZ)\right).
\]
This implies by induction over $i$ that there is an isomorphism of
$\IZ[1/p]$-modules
\[
H_m(B\IZ^n;\bfL(\Z))^{\IZ/p}[1/p]  \cong 
\bigoplus_{i=0}^n 
H_i(B\IZ^n_{\rho})^{\IZ/p}[1/p] \otimes_{\IZ} L_{m-i}(\IZ).
\]
Since $H_i(B\IZ^n_{\rho})^{\IZ/p} \cong \IZ^{r_i}$, the claim follows.
This finishes the proof of Theorem~\ref{the:L-infty(ZGamma)}.
\end{proof}


\subsection{Arbitrary decorations}
\label{subsec:arbitrary_decorations}

Finally we extend Theorem~\ref{the:L-infty(ZGamma)} to all decorations.  Recall we
abbreviate $\bfL^{\langle i \rangle}(\IZ)$ by $\bfL(\IZ)$ which is justified
by~\eqref{bfl(Z)_and_decos}.

\begin{theorem}[Computation of $L^{\langle i \rangle}(\IZ\Gamma)$] 
\label{the:L(ZGamma)_decorated}
Let $i \in \{2,1,0, -1, -2, \ldots\} \amalg \{-\infty\}$. Then:
\begin{enumerate}

\item\label{the:L(ZGamma)_decorated:iso_conjecture}
The assembly map
\[
A^{\langle i \rangle}_m \colon H_m^{\Gamma}(\eub{\Gamma};\bfL^{\langle i \rangle}) 
\xrightarrow{\cong}
H_m^{\Gamma}(\pt;\bfL^{\langle i \rangle}) = 
L^{\langle i \rangle}_m(\IZ \Gamma)
\]
is an isomorphism for $m \in \IZ$;

\item\label{the:L(ZGamma)_decorated:exact_sequence}
For $m \in \IZ$ there is an exact sequence
\[
0 \to \bigoplus_{(P) \in \calp} \widetilde{L}^{\langle i \rangle}_{m}(\IZ P) \to 
L^{\langle i\rangle}_{m}(\IZ \Gamma)  \xrightarrow{\beta_m^{\langle i \rangle}}
H_{m}(\bub{\Gamma};\bfL(\Z )) \to 0,
\]
which splits after inverting $p$. The first map is given by the various inclusions $P \to \Gamma$ and 
$\beta_m^{\langle i \rangle}$ is defined by the composite
\[\beta_m^{\langle i \rangle} \colon L^{\langle i\rangle}_{m}(\IZ \Gamma) \xrightarrow{(A^{\langle i \rangle}_m)^{-1}}
H_m^{\Gamma}(\eub{\Gamma};\bfL^{\langle i \rangle})  \xrightarrow{\ind_{\Gamma \to \{1\}}}
H_{m}(\bub{\Gamma};\bfL(\IZ));
\]

\item\label{the:L(ZGamma)_decorated:odd}
For $m \in \IZ$ the maps
\begin{eqnarray*}
L^{s}_{2m+1}(\IZ \Gamma)
& \xrightarrow{\cong} &
L^{\langle - \infty \rangle}_{2m+1}(\IZ \Gamma);
\\
\beta_{2m+1}^{\langle i \rangle} \colon 
L^{s}_{2m+1}(\IZ \Gamma) 
&\xrightarrow{\cong} &
H_{2m+1}(\bub{\Gamma};\bfL(\IZ)),
\end{eqnarray*}
are bijective and transferring to the finite index subgroup $
\Z^n$ of  $\Gamma$ 
induces an isomorphism
\[
L^{s}_{2m+1}(\IZ \Gamma) 
\xrightarrow{\cong} 
L_{2m+1}(\IZ [\IZ^n])^{\IZ/p};
\]

\item\label{the:L(ZGamma)_decorated:explicite}
We have 
\[
L^{s}_{m}(\IZ \Gamma)  \cong
\begin{cases}
\IZ^{p^k(p-1)/2} \oplus 
\left(\bigoplus_{i=0}^n  L_{m-i}(\IZ)^{r_{i}} \right)
& m \; \text{even};
\\
\bigoplus_{i=0}^n  L_{m-i}(\IZ)^{r_{i}}  & 
m \; \text{odd}.
\end{cases}
\]
\end{enumerate}
\end{theorem}
\begin{proof}~%
\ref{the:L(ZGamma)_decorated:iso_conjecture}
and~\ref{the:L(ZGamma)_decorated:exact_sequence}
Since $\Wh_i(P) = 0$ because of~\cite{Carter(1980)} and $\Wh_i(\Gamma) = 0$ by 
Theorem~\ref{the:algebraic_K-theory} for $i \le -2$, we obtain fron~\eqref{Rothenberg_sequence}
for $i \le -1$, $m \in \IZ$  isomorphisms
\begin{eqnarray*}
& L^{\langle -\infty \rangle}_m(\IZ P) 
\xrightarrow{\cong} 
L^{\langle i - 1 \rangle}_m(\IZ P) 
\xrightarrow{\cong} 
L^{\langle i \rangle}_m(\IZ P);
&
\\
& L^{\langle -\infty  \rangle}_m(\IZ \Gamma) 
\xrightarrow{\cong} 
L^{\langle i -1 \rangle}_m(\IZ \Gamma) 
\xrightarrow{\cong} 
L^{\langle i \rangle}_m(\IZ \Gamma);
\\
& H_m^{\Gamma}(\eub{\Gamma};\bfL^{\langle -\infty  \rangle})
\xrightarrow{\cong} 
H_m^{\Gamma}(\eub{\Gamma};\bfL^{\langle i -1   \rangle})
\xrightarrow{\cong} 
H_m^{\Gamma}(\eub{\Gamma};\bfL^{\langle i   \rangle}). &
\end{eqnarray*}
This  together with Theorem~\ref{the:L-infty(ZGamma)}
implies that assertions~\ref{the:L(ZGamma)_decorated:iso_conjecture} 
and~\ref{the:L(ZGamma)_decorated:exact_sequence} are true for 
$i \in \{-1, -2, \ldots\} \amalg \{-\infty\}$. 

It remains to prove assertions~\ref{the:L(ZGamma)_decorated:iso_conjecture} 
and~\ref{the:L(ZGamma)_decorated:exact_sequence} for
$i = 0,1,2$ what we will do inductively. So we want to show
for $i \le 2$ that assertions~\ref{the:L(ZGamma)_decorated:iso_conjecture} 
and~\ref{the:L(ZGamma)_decorated:exact_sequence} are true for $i$ if they are true for $i-1$.
Consider the following commutative diagram
\begin{small}
\[
\xymatrix@!C=11em{
\vdots \ar[d] & \vdots \ar[d] &  \vdots \ar[d]_-{\cong} 
\\
\bigoplus_{(P) \in \calp} \widetilde{L}^{\langle i-1 \rangle}_m(\IZ P) \ar[r] \ar[d]
&
L^{\langle i-1 \rangle}_m(\IZ \Gamma) \ar[r]^{\beta_m^{\langle i - 1\rangle}}\ar[d]
&
H_m(\bub{\Gamma};\bfL^{\langle i - 1\rangle}(\Z)) \ar[d]
\\
\bigoplus_{(P) \in \calp} \widehat{H}^m(\IZ/2;\Wh_{i-1}(P))  \ar[r]_-{\cong} \ar[d]
&
\widehat{H}^m(\IZ/2;\Wh_{i-1}(\Gamma))  \ar[r] \ar[d]
& 0 \ar[d]
\\
\bigoplus_{(P) \in \calp} \widetilde{L}^{\langle i \rangle}_{m-1}(\IZ P) \ar[r] \ar[d]
&
L^{\langle i \rangle}_{m-1}(\IZ \Gamma) \ar[r]^{\overline{\beta}_m^{\langle i \rangle}} \ar[d]
&
H_{m-1}(\bub{\Gamma};\bfL^{\langle i \rangle}(\Z)) \ar[d]_-{\cong}
\\
\bigoplus_{(P) \in \calp} \widetilde{L}^{\langle i-1 \rangle}_{m-1}(\IZ P) \ar[r] \ar[d]
&
L^{\langle i-1 \rangle}_{m-1}(\IZ \Gamma) \ar[r]^{\beta_m^{\langle i - 1\rangle}} \ar[d]
&
H_{m-1}(\bub{\Gamma};\bfL^{\langle i - 1\rangle}(\Z)) \ar[d]
\\
\bigoplus_{(P) \in \calp} \widehat{H}^{m-1}(\IZ/2;\Wh_{i-1}(P))  \ar[r]_-{\cong} \ar[d]
&
\widehat{H}^{m-1}(\IZ/2;\Wh_{i-1}(\Gamma))  \ar[r] \ar[d]
& 0 \ar[d]
\\
\vdots & \vdots & \vdots 
}
\]
\end{small}
The left vertical column is the direct sum over $\calp$ of the
Rothenberg sequences~\eqref{Rothenberg_sequence} associated to $P$.
The middle column is the Rothenberg sequences~\eqref{Rothenberg_sequence} 
associated to $\Gamma$. The isomorphism  in the right column come from~\eqref{bfl(Z)_and_decos}.
Notice that all columns are exact. In each row the left arrow comes from
the various inclusions $P \to \Gamma$. 
The rows involving the Tate cohomology are exact sequences
\[
0 \to 
\bigoplus_{(P) \in \calp} \widehat{H}^m(\IZ/2;\Wh_{i-1}(P)) 
\to
\widehat{H}^m(\IZ/2;\Wh_{i-1}(\Gamma))  
\to 0 \to 0.
\]
The sequences for the decorations $\langle i-1 \rangle $ are short exact sequences by
induction hypothesis
\[
0 \to
\bigoplus_{(P) \in \calp} \widetilde{L}^{\langle i-1 \rangle}_m(\IZ P) \to
L^{\langle i-1 \rangle}_m(\IZ \Gamma)  
\xrightarrow{\beta_m^{\langle i - 1\rangle}}
H_m(\bub{\Gamma};\bfL(\Z)) \to 0.
\]
The map 
$\overline{\beta}_m^{\langle i \rangle} \colon L^{\langle i \rangle}_{m}(\IZ \Gamma) \to H_{m}(\bub{\Gamma};\bfL(\IZ))$
is defined such that the diagram commutes. The remaining columns yield short exact sequences
\begin{eqnarray}
&  0 \to
\bigoplus_{(P) \in \calp} \widetilde{L}^{\langle i \rangle}_m(\IZ P) 
\to L^{\langle i \rangle}_m(\IZ \Gamma)  
\xrightarrow{\overline{\beta}_m^{\langle i \rangle}}
H_m(\bub{\Gamma};\bfL(\Z)) 
\to 0.
&
\label{short_exact_sequence_for_i}
\end{eqnarray}
Lemma~\ref{lem:MV_sequence} 
 shows that there is a long exact sequence
\begin{multline*}
\cdots \to H_{m+1}(\bub{\Gamma};\bfL(\IZ)) 
\to \bigoplus_{(P) \in \calp} \widetilde{L}^{\langle i \rangle}_m(\IZ P) 
\to H^{\Gamma}_{m}(\eub{\Gamma};\bfL^{\langle i \rangle})
\\
\to H_{m}(\bub{\Gamma};\bfL(\IZ)) 
\to \bigoplus_{(P) \in \calp} \widetilde{L}^{\langle i \rangle}_{m-1}(\IZ P) \to \cdots
\end{multline*} 
and that the $\IZ[1/p]$-map 
\[
H^{\Gamma}_{m}(\eub{\Gamma};\bfL^{\langle i  \rangle})[1/p]
\to
 H_{m}(\bub{\Gamma};\bfL(\IZ))[1/p] 
\]
is split surjective. We have already shown, 
see~\eqref{short_exact_sequence_for_i}, that the composite
\[
\bigoplus_{(P) \in \calp} \widetilde{L}^{\langle i \rangle}_m(\IZ P) 
\to H^{\Gamma}_{m}(\eub{\Gamma};\bfL^{\langle i \rangle})
\to L^{\langle i \rangle}_m(\IZ \Gamma)
\]
is injective. Hence the long exact sequence reduces to short exact sequences
\[
0 \to \bigoplus_{(P) \in \calp} \widetilde{L}^{\langle i \rangle}_m(\IZ P) 
\to H^{\Gamma}_{m}(\eub{\Gamma};\bfL^{\langle i \rangle})
\to H_{m}(\bub{\Gamma};\bfL(\IZ))
\to 0
\]
which split after inverting $p$. We obtain the following commutative diagram
\[
\xymatrix{
0 \ar[r] 
&
\bigoplus_{(P) \in \calp} \widetilde{L}^{\langle i \rangle}_m(\IZ P) 
\ar[r] \ar[d]_-{\id} 
&
 H^{\Gamma}_{m}(\eub{\Gamma};\bfL^{\langle i \rangle})
\ar[r] \ar[d]^{A^{\langle i \rangle}_m} 
&
H_{m}(\bub{\Gamma};\bfL(\IZ))
\ar[r] \ar[d]_-{\id}
& 0
\\
0 \ar[r] 
&
\bigoplus_{(P) \in \calp} \widetilde{L}^{\langle i \rangle}_m(\IZ P)
\ar[r] 
&
L^{\langle i \rangle}_m(\IZ \Gamma) 
\ar[r]  
&
H_{m}(\bub{\Gamma};\bfL(\IZ))
\ar[r]  
& 0.
}
\]
Since the rows are exact and the first and third vertical arrows are
bijective, the middle arrow is bijective by the Five-Lemma.
This finishes the proof of 
assertions~\ref{the:L(ZGamma)_decorated:iso_conjecture}.

A direct computation shows that the map $\overline{\beta}_m^{\langle i \rangle}$ agrees with the map
$\beta_m^{\langle i \rangle}$. Now assertion~\ref{the:L(ZGamma)_decorated:exact_sequence} 
follows from~\eqref{short_exact_sequence_for_i}.
\\[1mm]~%
\ref{the:L(ZGamma)_decorated:odd}
The following isomorphism is due independently to Bak~\cite{Bak(1974)} and Wall~\cite[Corollary~2.4.3]{Wall(1976)}
\begin{eqnarray}
\widetilde{L}^s_m(\IZ P) \cong
\begin{cases} \label{equation:LsZ/p}
\IZ^{(p-1)/2} & m\; \text{even};
\\
0 & m\; \text{odd}.
\end{cases}
\label{widetildeLs(ZP)}
\end{eqnarray}
Hence assertion~\ref{the:L(ZGamma)_decorated:exact_sequence} implies that we obtain an 
isomorphism
\[
L^{s}_{2m+1}(\IZ \Gamma) 
\xrightarrow{\cong} 
H_{2m+1}(\bub{\Gamma};\bfL(\IZ)).
\]
The following diagram commutes
\[
\xymatrix{
L^{s}_{2m+1}(\IZ \Gamma) \ar[d] \ar[r]
&
H_{2m+1}(\bub{\Gamma};\bfL(\IZ)) \ar[d]^{\id}
\\
L^{\langle - \infty \rangle}_{2m+1}(\IZ \Gamma) \ar[r]
&
H_{2m+1}(\bub{\Gamma};\bfL(\IZ)).
}
\]
The lower horizontal arrow is an isomorphism by
Theorem~\ref{the:L-infty(ZGamma)}~%
\ref{the:L-infty(ZGamma):L2m_plus_1_abstract} and the right vertical arrow is an isomorphism 
by~\eqref{bfl(Z)_and_decos}. Hence the map
\[
L^{s}_{2m+1}(\IZ \Gamma)
\xrightarrow{\cong} 
L^{\langle - \infty \rangle}_{2m+1}(\IZ \Gamma)
\]
is an isomorphism. The following diagram commutes  
\[
\xymatrix{L^{s}_{2m+1}(\IZ \Gamma) \ar[d] \ar[r]
& 
L^{s}_{2m+1}(\IZ [\IZ^n])^{\IZ/p} \ar[d]
\\
L^{\langle - \infty \rangle}_{2m+1}(\IZ \Gamma) \ar[r]
& 
L^{\langle - \infty \rangle}_{2m+1}(\IZ [\IZ^n])^{\IZ/p}.
}
\] 
The lower horizontal arrow is an isomorphism by
Theorem~\ref{the:L-infty(ZGamma)}~%
\ref{the:L-infty(ZGamma):L2m_plus_1_abstract}.
Since $\Wh_i(\IZ^n) = 0$ for $i \le 1$, the right vertical arrow
is an isomorphism by the Rothenberg sequence~\eqref{Rothenberg_sequence}.
Hence the upper horizontal arrow is an isomorphism. This finishes
the proof of assertion~\ref{the:L(ZGamma)_decorated:odd}.
\\[1mm]~%
\ref{the:L(ZGamma)_decorated:explicite}
It suffices to prove the claim after inverting $2$ and after inverting $p$.
If we invert $2$, the natural comparison maps between the various decorated
$L$-groups are isomorphisms by the Rothenberg sequence~\eqref{Rothenberg_sequence}
and the claim follows from Theorem~\ref{the:L-infty(ZGamma)}~%
\ref{the:L-infty(ZGamma):explicite}.
If we invert $p$, the claim follows
from assertion~\ref{the:L(ZGamma)_decorated:exact_sequence}, Lemma~\ref{lem:calh(bub(Gamma))_to_calh(Zn)Z/p}
and isomorphisms~\eqref{l(Z)_and_decos},~%
\eqref{H_m(BZn;Llangle-inftyrangle)Z/p[1/p]}, and~\eqref{widetildeLs(ZP)}.
\end{proof}


\typeout{------------------------------ Structure sets ---------------------------}

\section{Structure sets}
\label{subsec:structure_sets}

Given a closed oriented $m$-dimensional manifold $N$,
we denote its \emph{geometric simple structure set} by $\strg(N)$. 
An element is represented by a simple homotopy equivalence
$N' \to N$ with a closed  manifold $N'$ as source and $N$ as target.
Two such maps $g' \colon N' \to N$ and $g'' \colon N'' \to N$ define the same element
in $\strg(N)$ if and only if there is a  homeomorphism
$u \colon N' \to N''$ such that $g'' \circ u$ and $g'$ are homotopic. 
This structure set appears in the geometric surgery exact sequence due to Browder, Novikov,
Sullivan and Wall, see~\cite[Theorem 10.3]{Wall(1999)} and~\cite[Chap.~5]{Lueck(2002)}, valid when $m = \dim N \geq 5$.
\begin{multline}
\cdots \to \caln(N \times (D^1,S^0)) \to L_{m+1}^s(\IZ[\pi_1(N)]) \to
\strg(N) 
\\
\to \caln(N) \to L_{m}^s(\IZ[\pi_1(N)]).
\label{geometric_surgery_sequence}
\end{multline}

In the sequel we abbreviate
\[
\bfL := \bfL^s = \bfL^{\langle  2 \rangle} \colon\GROUPOIDS \to \SPECTRA
\]
so that we have $\pi_m(\bfL(\ol{G/H})) = L^s_m(\Z H)$ for a group $G$ and subgroup $H \subseteq G$.

Given any $CW$-complex $X$, 
there is an exact algebraic surgery sequence of abelian groups
\begin{multline}
\cdots \xrightarrow{\eta_{m+2}(X)}  H_{m+1}(X;\bfL(\Z)) \xrightarrow{A_{m+1}(X)} L_{m+1}^s(\IZ[\pi_1(X)]) 
\xrightarrow{\xi_{m+1}(X)} \strp_{m+1}(X) 
\\
\xrightarrow{\eta_{m+1}(X)} H_{m}(X;\bfL(\Z)) \xrightarrow{A_{m}(X)} L_{m}^s(\IZ[\pi_1(X)]) 
\xrightarrow{\xi_{m}(X)} \cdots,
\label{algebraic_surgery_sequence}
\end{multline}
natural in $X$.  There is also a $1$-connective version of the sequence~\eqref{algebraic_surgery_sequence}
\begin{multline}
\cdots \xrightarrow{\eta_{m+2}^{\langle 1 \rangle}(X)}  H_{m+1}(X;\bfL\langle 1 \rangle) 
\xrightarrow{A_{m+1}^{\langle 1 \rangle}(X)} L_{m+1}^s(\IZ[\pi_1(X)]) 
\xrightarrow{\xi_{m+1}^{\langle 1 \rangle}(X)} \strone_{m+1}(X) 
\\
\xrightarrow{\eta_{m+1}^{\langle 1 \rangle}(X)} H_{m}(X;\bfL\langle 1 \rangle) 
\xrightarrow{A_{m}^{\langle 1 \rangle}(X)} L_{m}^s(\IZ[\pi_1(X)]) 
\xrightarrow{\xi^{\langle 1 \rangle}_{m}(X)} \cdots,
\label{algebraic_surgery_sequence_one_connective}
\end{multline}
where $\bfp \colon \bfL \langle 1 \rangle \to \bfL$ is the $1$-connective cover of the
$L$-theory spectrum $\bfL := \bfL(\Z)$. Here
$\pi_i(\bfp) \colon \pi_i(\bfL \langle 1 \rangle) \to \pi_i(\bfL(\Z))$ 
is an isomorphism for $i \ge 1$
and $\pi_i(\bfL \langle 1 \rangle) = 0$ for $i \le 0$.

The algebraic surgery exact sequences can be constructed in two ways.  It can be
constructed by defining the structure groups to be the homotopy groups of the cofiber of a
spectrum-level assembly map (defined, for example, in~\cite[Example
5.5]{Davis-Lueck(1998)}) or at the level of representatives (see the quadratic structure
group ${\mathbb S}_{m}(\IZ,X)$ appearing in~\cite[Definition~14.6 on
page~148]{Ranicki(1992)}).  Using the second definition, Ranicki identified the
geometric surgery sequence~\eqref{geometric_surgery_sequence} with the $1$-connective
algebraic surgery sequence~\eqref{algebraic_surgery_sequence_one_connective} truncated at
$L_5^s(\Z[\pi_1(X)]$, see~\cite[Theorem~18.5 on page~198]{Ranicki(1992)} and~\cite{Kuehl-Macko-Mole(2013)}.
In particular we get an identification
\begin{equation}
\strg(N)  \cong  \strone_{n+l+1}(N),
\label{ident_structure_sets}
\end{equation}
and a map
\begin{eqnarray}
j(N) \colon \strg(N) \cong \strone_{n+l+1}(N) 
& \to & 
\strp_{n+l+1}(N)
\label{j(N)}
\end{eqnarray}
from the canonical map $\bfp \colon \bfL \langle 1 \rangle \to \bfL$.


\typeout{--The periodic simple structure set of \texorpdfstring{$BP$}{BP} for a finite \texorpdfstring{$p$}{p}-group -----}


\section{The periodic simple structure set of \texorpdfstring{$BP$}{BP} for a finite \texorpdfstring{$p$}{p}-group}
\label{subsec:The_periodic_simple_structure_set_of_BG_finite_p-group}

In this section we compute the periodic simple structure groups $\strp_{*}(BP)$ for a
finite $p$-group $P$ and $p$ an odd prime.

Let $k$ be a nonzero integer.  A map of abelian groups $f\colon  A \to B$ is a
\emph{$1/k$-equivalence} if $f \otimes \id \colon  A \otimes \Z[1/k] \to B \otimes \Z[1/k]$ is an
isomorphism.  An abelian group $A$ is \emph{$1/k$-local} if the map
$A \otimes \Z \to A \otimes \Z[1/k]$ is an isomorphism.  A map $A \to B$ is a
\emph{$1/k$-localization} if it is a $1/k$-equivalence and $B$ is $1/k$-local.  If $k$ and
$l$ are relatively prime integers, then $A$ is $1/k$-local if and only if
$A \otimes \Z[1/l]$ is $1/k$-local. 

\begin{theorem}[The periodic structure set of $BP$ for a finite $p$-group $P$]
\label{the:The_periodic_structure_set_of_BP_for_a_finite_p-group}
  Let $p$ be an odd prime and $P$ be a finite $p$-group.  Let $\alpha(P; \C)$ be the
  number of irreducible real representations of $P$ of complex type.

\begin{enumerate}
\item\label{the:The_periodic_structure_set_of_BP_for_a_finite_p-group:point} 
 \[
 \wL^s_m(\Z P) 
 \cong \begin{cases}
\Z^{\alpha(P; \C)} & m \;\text{even},
\\
0  & m \; \text{odd},
\end{cases}
\]
where $\wL^s_m(\Z P)$ is the cokernel of the map $L^s_m(\Z ) \to L^s_m(\Z P)$ induced by the inclusion $1 \to P$.

\item\label{the:The_periodic_structure_set_of_BP_for_a_finite_p-group:localization} The
  homomorphism $\xi_*(BP) \colon L^s_*(\Z P) \to \strp_*(BP)$ induces a $1/p$-locali\-za\-tion
  $\widetilde{\xi}_*(BP) \colon\wL^s_*(\Z P) \to \strp_*(BP)$.

\end{enumerate}

\end{theorem}

\begin{proof}~%
\ref{the:The_periodic_structure_set_of_BP_for_a_finite_p-group:point} This a result
  due independently to Bak and Wall, see~\cite{Bak(1974)} and~\cite[Corollary~2.4.3]{Wall(1976)}.  
\\[1mm]~%
\ref{the:The_periodic_structure_set_of_BP_for_a_finite_p-group:localization} The
  assembly map $A_*(\pt) \colon H_*(\pt; \bfL(\Z)) \to L_*^s(\Z)$ is an isomorphism, essentially by
  definition.  Thus there is are reduced algebraic surgery exact sequence
\begin{multline}
\cdots \xrightarrow{\wt \eta_{m+1}(X)}  \wt H_{m}(X;\bfL(\Z)) \xrightarrow{\wt A_{m}(X)} \wL_{m}^s(\IZ[\pi_1(X)]) 
\xrightarrow{\wt \xi_{m}(X)} \strp_m(X) 
\\
\xrightarrow{\wt \eta_m(X)} \wt H_{m-1}(X;\bfL(\Z)) \xrightarrow{\wt A_{m-1}(X)} \wt L_{m-1}^s(\IZ[\pi_1(X)]) 
\xrightarrow{\wt \xi_{m-1}(X)} \cdots,
\label{reduced_surgery_sequence_periodic}
\end{multline}
and
\begin{multline}
\cdots \xrightarrow{\wt \eta_{m+1}^{\langle 1 \rangle}(X)}  \wt H_{m}(X;\bfL\langle 1 \rangle) 
\xrightarrow{\wt A_{m}^{\langle 1 \rangle}(X)} \wL_{m}^s(\IZ[\pi_1(X)]) 
\xrightarrow{\wt \xi^{\langle 1 \rangle}_{m}(X)} \strg(X) 
\\
\xrightarrow{\wt \eta^{\langle 1 \rangle}_m(X)} \wt H_{m-1}(X;\bfL \langle 1 \rangle) 
\xrightarrow{\wt A_{m-1}^{\langle 1 \rangle}(X)} \wt L_{m-1}^s(\IZ[\pi_1(X)]) 
\xrightarrow{\wt \xi_{m-1}(X)} \cdots.
\label{reduced_surgery_sequence_geometric}
\end{multline}

Using the Atiyah-Hirzebruch spectral sequence and the $\Z$-flatness of $\Z[1/p]$, one sees
that $\wt H_*(BP; \bfL(\Z)) \otimes \Z[1/p] = 0$.  Using the flatness of $\Z[1/p]$ again it
follows that $\wt \xi_*(BP)$ is a $1/p$-equivalence.  Thus it suffices to show that
$\strp_*(BP)$ is $1/p$-local. 

We will express the structure set in terms of an equivariant homology theory.  
We get from~\cite[Theorem B.1]{Connolly-Davis-Khan(2014H1)} an isomorphism
\[
\strp_m(X) \cong  H^G_m(\wt X\to\pt;  \bfL(\Z))
\]
for a connected CW-complex $X$ with fundamental group $G$.  
Thus it suffices to show that $H^P_*(EP \to\pt;  \bfL(\Z))$ is $1/p$-local.
Since $p$ is odd, it remains to show that
$H^P_*(EP \to\pt; \bfL(\Z))[1/2]$ is $1/p$-local since $p$ is odd.
Because of Theorem~\ref{the:Comparing_L-theory_and_topological_K-theory} it is enough to show
that $KO_*(EP \to\pt; \bfL(\Z))[1/2]$ is $1/p$-local.

The composite of complexification $c$ with the forgetful map $r$
\[
{\bf KO} \xrightarrow{c} {\bf K} \xrightarrow{r} {\bf KO}
\]
is multiplication by 2, so $KO_m^P(EP \to \pt)[1/2]$ is a summand of
$K_m^P(EP \to \pt)[1/2]$.  Since  a summand of a $1/p$-local abelian group is $1/p$-local,
it suffices to show that $K_*^P(EP \to \pt)$ is $1/p$-local. Hence the proof of
Theorem~\ref{the:The_periodic_structure_set_of_BP_for_a_finite_p-group} is finished
after we have proved the next result.
\end{proof}

\begin{lemma}  Let $p$ be a prime and $P$ be a $p$-group.  Let $\alpha(P; \C)$ 
be the number of irreducible complex representations of $P$.
\label{lemma:relative_K-theory}
\begin{enumerate}
\item\label{lemma:relative_K-theory:point} $\wK_m^P(\pt) 
 \cong \begin{cases}
\Z^{\alpha(P; \C)-1} & m \;\text{even},
\\
0  & m \; \text{odd}.
\end{cases}$

\item\label{lemma:relative_K-theory:localization}

The map $\wK_*^P(\pt)   \longrightarrow K^P_*(EP \to \pt)$ is a $1/p$-localization.

\end{enumerate}
\end{lemma}

\begin{proof}~%
\ref{lemma:relative_K-theory:point} Let $R(P)$ be the complex representation ring.  Note
  that equivariant $K$-cohomology and equivariant $K$-homology are 2-periodic,
  $K^0_P(\pt) = R(P)$, $K_0^P(\pt) = \Hom_{\Z}(R(P),\Z)$, and
  $K^1_P(\pt) = 0 = K_1^P(\pt)$.
  \\[1mm]~%
\ref{lemma:relative_K-theory:localization} Our method for computing the relative
  equivariant $K$-homology will be to apply a Universal Coefficient Theorem.  This
  Universal Coefficient Theorem applies only to finite complexes.  Hence we take models of
  $BP$ and $EP$ whose $n$-skeletons $BP^n$ and $EP^n$ have a finite number of cells.  For
  each $n$ we have the exact sequence
\begin{multline}
0 \to K_P^1(EP^n)  \to K_P^0(EP^n \to \pt) \to K^0_P(\pt)\\
\to K_P^0(EP^n) \to K_P^1(EP^n \to \pt) \to 0.
\end{multline}
It induces an exact sequence of pro-$\IZ$-modules indexed by $n = 0,1,2,3 \ldots$
\begin{multline}\label{equivariant_K_cohomology}
0 \to \{K_P^1(EP^n)\}  \to \{K_P^0(EP^n \to \pt)\} \to \{K^0_P(\pt)\} 
\\
\to \{K_P^0(EP^n)\} \to \{K_P^1(EP^n \to \pt)\} \to 0,
\end{multline}
where $\{K^0_P(\pt)\}$ is a constant pro-$\Z$-module. For an introduction to the abelian
category of pro-$\IZ$-modules see~\cite{Atiyah-Segal(1969)}.  By the Atiyah-Segal
Completion Theorem~\cite{Atiyah-Segal(1969)}, there is an isomorphism of pro-$\Z$-modules
$ \{K_P^1(EP^n)\} \cong \{0\} $ and a commutative diagram of pro-$\IZ$-modules
\[
\xymatrix{
\{R(P)\} \ar[r]\ar[d]^{\cong}  & \{R(P)/I^n\} \ar[d]^{\cong} 
\\
\{K_P^0(\pt)\} \ar[r] & \{K_P^0(EP^n)\} 
,}
\]
where $I = \ker(R(P) \to \Z)$ is the augmentation ideal of the representation ring.
It follows that the sequence~\eqref{equivariant_K_cohomology} can be identified with
\[ 
0\to \{0\} \to \{I^n\}  \to \{R(P)\} 
\to \{R(P)/I^n\} \to \{0\} \to 0,
\]
and in particular,
\begin{eqnarray*}
\{K_P^1(EP^n \to \pt)\} & = & \{0\};
\\
\{K_P^0(EP^n \to \pt)\} & = & \{I^n\}.
\end{eqnarray*}

A Universal Coefficient Theorem for $K_*^P$ proven in~\cite{Joachim-Lueck(2013)} states
that for an equivariant map $X \to Y$ of finite $P$-CW-complexes, there is a short exact
sequence
\[
0 \to \Ext_{\Z}(K_P^{*+1}(X \to Y),\Z) \to K_*^P(X \to Y)\to  \Hom_{\Z}(K_P^{*}(X \to Y),\Z) \to 0.
\]
It is an algebraic fact that $\{I^n\} = \{p^nI\}$ and, setting
$I^* = \Hom_{\Z}(I,\Z) \subset I^* \otimes \Q$, that $(p^nI)^* = p^{-n}I^*$.  Thus
\begin{align*}
K^P_0(EP \to \pt) & \cong \colim_{n \to \infty}  K^P_0(EP^n \to \pt) \\
& \cong  \colim_{n \to \infty} \Hom_{\IZ}(K_P^0(EP^n \to \pt),\IZ) \\
& \cong  \colim_{n \to \infty} (p^nI)^* \\
&  \cong  I^* \otimes_{\IZ} \IZ[1/p].
\end{align*}
The map $\wK^P_0( \pt) \to K^P_0(EP \to \pt)$ can be identified with the inclusion
$I^* = I^* \otimes \Z \to I^* \otimes \Z[1/p]$.  Similar, but easier reasoning shows that
$K^P_1(EP \to \pt)=0$. Hence Lemma~\ref{lemma:relative_K-theory} follows.
\end{proof}


\typeout{------------- The periodic simple structure set of \texorpdfstring{$B\Gamma$}{BGamma}  ----------------}


\section{The periodic simple structure set of \texorpdfstring{$B\Gamma$}{BGamma}}
\label{sec:The_periodic_simple_structure_set_of_BGamma}

In this section we compute the periodic simple structure set of $B\Gamma$.  Recall that
$\calp$ is the set of conjugacy classes of subgroups of $\Gamma$ of order $p$ and that the
order of $\calp$ is $p^k$.

\begin{theorem}[Computation of $\strp(B\Gamma)$]
\label{the:Computation_of_strp(BGamma)}\ 
\begin{enumerate}

\item\label{the:Computation_of_strp(BGamma):induction_iso}
The map
\[\bigoplus_{(P) \in \calp} \strp_m(BP) \xrightarrow{\cong}
\strp_m(B\Gamma)\]
induced by the various inclusions $P \to \Gamma$ is for all $m \in
\IZ$ an isomorphism.  

We get
\[\strp_{m}(B\Gamma) \cong \begin{cases}
\big(\IZ[1/p]\bigr)^{p^k \cdot (p-1)/2} & m \;\text{odd};
\\
0  & m \; \text{even};
\end{cases}
\]

\item\label{the:Computation_of_strp(BGamma):restriction_iso} Let
  $\pr \colon \Gamma \to \Gamma_{\ab} = \Gamma/[\Gamma,\Gamma]$ be the projection onto the
  abelianization of $\Gamma$.  Let
  $\strp_m(B\!\pr) \colon \strp_m(B\Gamma) \to \strp_m(\Gamma_{\ab})$ be the induced map.
  Recall $\Gamma_{\ab} \cong (\Z/p)^{k+1}$.  For every $(P) \in \calp$, let
  $P' \subset \Gamma_{\ab}$ be the image of $P$ under $\pr$.  Let
\[
\res_{\Gamma_{\ab}}^{P'} \colon \strp_m(B\Gamma_{\ab}) \to \strp_m(BP')
\]
be the map induced by transfer to the subgroup $P' \subset \Gamma_{\ab}$.

Then the map
\[ 
\prod_{(P) \in \calp} \res_{\Gamma_{\ab}}^{P'} \circ \ \strp_m(B\!\pr) 
\colon \strp_m(B\Gamma) 
\xrightarrow{\cong} \prod_{(P) \in \calp} \strp_m(BP')
\]
is an isomorphism.

\end{enumerate}
\end{theorem}    
\begin{proof}~%
\ref{the:Computation_of_strp(BGamma):induction_iso}
By~\cite[Theorem B.1]{Connolly-Davis-Khan(2014H1)},
we have the identifications
\begin{eqnarray*}
\strp_m(B\Gamma) & = & H^{\Gamma}_{m}(E\Gamma \to \pt; \bfL^s);
\\
\strp_m(BP) & = & H^{P}_{m}(EP \to \pt; \bfL^s).
\end{eqnarray*}
Theorem~\ref{the:L(ZGamma)_decorated}~\ref{the:L(ZGamma)_decorated:iso_conjecture} implies
\begin{eqnarray*}
H^{\Gamma}_{*}(E\Gamma \to \eub{\Gamma}; \bfL^s) 
& = & 
H^{\Gamma}_{*}(E\Gamma \to \pt; \bfL^s).
\end{eqnarray*}

We get for $m \in \IZ$ from the induction structure, see~\cite[Section~1]{Lueck(2002b)}  isomorphisms
\begin{equation}\label{quotient_identification}
  H_m^{\{1\}}(BP;\bfL^s) \xleftarrow{\cong} H^P_m( EP ;  \bfL^s)
  \xrightarrow{\cong}   H_m^\Gamma( \Gamma \times_P EP;  \bfL^s).
\end{equation}

We conclude from the $\Gamma$-pushout~\eqref{G-pushoutfor_EGamma_to_eunderbar_Gamma} and
the identification~\eqref{quotient_identification},
\begin{eqnarray*}
\bigoplus_{(P) \in \calp} H^{P}_{*}(EP \to \pt; \bfL^s)
& = & 
H^{\Gamma}_{*}(E\Gamma \to \eub{\Gamma}; \bfL^s).
\end{eqnarray*}
Hence 
\[
\bigoplus_{(P) \in \calp} \strp_m(BP) \xrightarrow{\cong}
\strp_m(B\Gamma) 
\]
is an isomorphisms. 
Now assertion~\ref{the:Computation_of_strp(BGamma):induction_iso} follows
from Theorem~\ref{the:The_periodic_structure_set_of_BP_for_a_finite_p-group} and 
Lemma~\ref{lem:preliminaries_about_Gamma_and_Zn_rho}~\ref{lem:preliminaries_about_Gamma_and_Zn_rho:order_of_calp}.
\\[1mm]~%
\ref{the:Computation_of_strp(BGamma):restriction_iso}
Because of assertion~\ref{the:Computation_of_strp(BGamma):induction_iso}
it suffices to show that the composite
\[\bigoplus_{(P) \in \calp} \strp_m(BP) \xrightarrow{\cong}
\strp_m(B\Gamma)  \xrightarrow{\prod_{(P) \in \calp} \res_{\Gamma_{\ab}}^{P'} \circ \ \strp_m(B\!\pr)}
\prod_{(P) \in \calp} \strp_m(BP')
\]
is an isomorphism. We conclude from Theorem~\ref{the:The_periodic_structure_set_of_BP_for_a_finite_p-group}~%
\ref{the:The_periodic_structure_set_of_BP_for_a_finite_p-group:localization} that it suffices to show that
the composite 
\[
\bigoplus_{(P) \in \calp} \widetilde{L}^s_{m}(\IZ P) \to L^s_m(\IZ \Gamma) \xrightarrow{\pr_*}
\widetilde{L}^s_m(\IZ[\Gamma_{\ab}]) \xrightarrow{\prod_{(P) \in \calp} \res_{\Gamma_{\ab}}^{P'}}
\prod_{(P) \in \calp} \widetilde{L}^s_{m}(\IZ P')
\] 
is an isomorphism after inverting $p$, where the first map is given by induction with the
various inclusions $P \to \Gamma$, the second by induction with 
$\pr \colon \Gamma \to \Gamma_{\ab}$ and the third is the product over the various transfer
homomorphisms $\res_{\Gamma_{\ab}}^{P'}$. This composite
agrees with the composite
\[
\bigoplus_{(P) \in \calp} \widetilde{L}^s_{m}(\IZ P') 
\xrightarrow{\left(\prod_{(Q) \in \calp} \res_{\Gamma_{\ab}}^{Q'}\right) 
\circ \left(\bigoplus_{(P) \in \calp} \ind_{P'}^{\Gamma_{\ab}}\right)}
\prod_{(Q) \in \calp} \widetilde{L}^s_{m}(\IZ Q').
\]
Hence it suffices to show
that for $(P),(Q) \in \calp$ 
\[
\res_{\Gamma_{\ab}}^{Q'} \circ \ind_{P'}^{\Gamma_{\ab}} \colon
\widetilde{L}^s_{m}(\IZ P') \to \widetilde{L}^s_{m}(\IZ Q')
\]
is trivial for $(P) \not= (Q)$ and $p^m \cdot \id$ for
$(P) = (Q)$. Notice that $P' = Q' \Leftrightarrow (P) = (Q)$ 
holds for $(P), (Q) \in \calp$
by Lemma~\ref{lem:preliminaries_about_Gamma_and_Zn_rho}~%
\ref{lem:preliminaries_about_Gamma_and_Zn_rho:list_of_finite_subgroups} 
and~\ref{lem:preliminaries_about_Gamma_and_Zn_rho:abelianization}.
By the double coset formula the composite
\[
\res^{\Gamma_{\ab}}_{Q'} \circ \ind_{P'}^{\Gamma_{\ab}} \colon
L^s_{m}(\IZ P') \to L^s_{m}(\IZ Q')
\]
factorizes through $L^s_m(\IZ)$ if $P' \not= Q'$, and  is 
$\sum_{\Gamma_{\ab}/P} \id$ if $P'= Q'$.
Since $\Gamma_{\ab}/P$ contains $p^k$ elements by 
Lemma~\ref{lem:preliminaries_about_Gamma_and_Zn_rho}~%
\ref{lem:preliminaries_about_Gamma_and_Zn_rho:list_of_finite_subgroups} 
and~\ref{lem:preliminaries_about_Gamma_and_Zn_rho:abelianization},
Theorem~\ref{the:Computation_of_strp(BGamma)} follows.

\end{proof}

The splitting of the periodic structure set appearing in 
Theorem~\ref{the:Computation_of_strp(BGamma)}~%
\ref{the:Computation_of_strp(BGamma):induction_iso}
has already been established in~\cite[Theorem~2.12]{Lueck-Rosenthal(2014)}.


\typeout{-------------- The periodic simple structure set of M  -------------------------}


\section{The periodic simple structure set of $M$}
\label{sec:The_periodic_simple_structure_set_of_M}

In this section we compute the periodic simple structure set $\strp_{n+l+1}(M)$ of $M$.
Recall that $M = T^n_{\rho} \times_{\Z/p} S^l$ and $\pi_1(M) =\Gamma = \Z^n \times_{\rho} \Z/p$.
Let $f \colon M \to B\Gamma$ be a classifying map for the universal
covering of $M$.   

\begin{theorem}[The periodic simple structure set of $M$]\
\label{the:The_periodic_simple_structure_set_of_M}
There is a homomorphism
\[
\sigma \colon \strp_{n+l + 1}(M) \to H_n(T^n;\bfL(\Z))^{\Z/p}
\]
such that the following holds:
\begin{enumerate}
\item\label{the:The_periodic_simple_structure_set_of_M:inj_map_strp}
The map
\[
\sigma \times \strp_{n+l+1}(f) \colon \strp_{n+l+1}(M) \to 
H_n(T^n;\bfL(\Z))^{\Z/p}  \times \strp_{n+l+1}(B\Gamma)
\] 
is injective;

\item\label{the:The_periodic_simple_structure_set_of_M:image}
The cokernel of $\sigma$ is a finite abelian $p$-group;

\item\label{the:The_periodic_simple_structure_set_of_M:kernel}
Consider the composite 
\[
\nu \colon \bigoplus_{(P) \in \calp} \widetilde{L}^s_{n+l+1}(\IZ P) \to 
\widetilde{L}^s_{n+l+1}(\IZ \Gamma)
\xrightarrow{\widetilde{\xi}_{n+l+1}(M)} \strp_{n+l+1}(M)
\]
where the first map is given by induction with the various inclusions $P \to \Gamma$
and $\widetilde{\xi}_{n+l+1}(M)$ comes from~\eqref{reduced_surgery_sequence_periodic}.

Then $\nu$ is injective, the image of $\nu$ is contained in the kernel of $\sigma$ and
$\ker(\sigma)/im(\nu)$ is a finite abelian $p$-group;

\item\label{the:The_periodic_simple_structure_set_of_M:inverting_p}
After inverting $p$ we obtain an isomorphism
\[
\left(\sigma \times \strp_{n+l+1}(f)\right)[1/p] \colon \strp_{n+l+1}(M)[1/p] \to 
H_n(T^n;\bfL(\Z))^{\Z/p}[1/p]  \times \strp_{n+l+1}(B\Gamma)[1/p];
\]

\item\label{the:The_periodic_simple_structure_set_of_M:comp}
As an abelian group we have
\begin{eqnarray*}
\strp_{n+l+1}(M) &\cong & 
\IZ^{p^k(p-1)/2} \oplus \bigoplus_{i = 0}^{n} L_{n-i}(\IZ)^{r_i},
\end{eqnarray*}
where the numbers $r_i$ are defined in~\eqref{r-numbers}.
\end{enumerate}
\end{theorem}

\begin{remark}  \label{remark:points_of_view_sigma}
 There are several different points of views on the codomain of $\sigma$.   Indeed there are isomorphisms
 \[
 L_n(\Z[\Z^n])^{\Z/p} \cong H_n(T^n;\bfL(\Z))^{\Z/p} \cong \bigoplus_{i = 0}^{n} (H_i(T^n)^{\Z/p} \otimes L_{n-i}(\Z))  \cong \bigoplus_{i = 0}^{n} L_{n-i}(\Z)^{r_i}.
 \]
 The first isomorphism is due to the Shaneson splitting/Farrell-Jones Conjecture, the second isomorphism is due to the collapse of the Atiyah-Hirzebruch Spectral Sequence, and the last isomorphism comes from~\eqref{r-numbers}.
\end{remark}

In the proof of Theorem~\ref{the:The_periodic_simple_structure_set_of_M} it will be
convenient to define $\mu \colon \strp_{n+l + 1}(M) \to L_n(\IZ[\IZ^n_{\rho}])^{\IZ/p}$ and then define the composite
$
\sigma \colon \strp_{n+l + 1}(M) \xrightarrow{\mu} L_n(\IZ[\IZ^n_{\rho}])^{\IZ/p} \cong H_n(T^n;\bfL(\Z))^{\Z/p} $, noting that $\ker \mu = \ker \sigma$ and $\cok \mu \cong \cok \sigma$.
\begin{proof}[Proof of assertion~\ref{the:The_periodic_simple_structure_set_of_M:inj_map_strp} 
of Theorem~\ref{the:The_periodic_simple_structure_set_of_M}]
We will have proved
 assertion~\ref{the:The_periodic_simple_structure_set_of_M:inj_map_strp} once we accomplish the following two goals.

\begin{itemize}
\item[Goal 1):] Construct the commutative diagram~\eqref{big_commutative_surgery_diagram} below.

\item[Goal 2):] Show that the induced maps
\[
\ker \bigl(\strp_{n+l+1}(f)\bigr) \to \ker \bigl(H_{n+l}(f;\bfL(\Z))\bigr) \to \ker \bigl(E^{\infty}_{l,n}(f)\bigr)
\]
are isomorphisms.  
\end{itemize}

Once we accomplish these two goals, a diagram chase gives the proof of 
Theorem~\ref{the:The_periodic_simple_structure_set_of_M}~%
\ref{the:The_periodic_simple_structure_set_of_M:inj_map_strp}.

We now construct the following commutative diagram of abelian groups.
\begin{eqnarray}
\label{big_commutative_surgery_diagram}
& & 
\\
\nonumber
\xymatrix@!C=14em{
\strp_{n+l+1}(M) \ar[r]^-{\strp_{n+l+1}(f)}
\ar[d]^{\eta_{n+l+1}(M)} 
\ar@/_{28mm}/[ddddddd]^{\mu}
&
\strp_{n+l+1}(B\Gamma) 
\ar[d]^{\eta_{n+l+1}(B\Gamma)}
\\
H_{n+l}(M;\bfL(\Z)) \ar[r]^-{H_{n+l}(f;\bfL(\Z))} 
&
H_{n+l}(B\Gamma;\bfL(\Z)) 
\\
F_{l,n}(M) \ar[u]_{\inc}^{\cong} \ar[r]^-{F_{l,n}(f)}  \ar[d]^{\pr}
&
F_{l,n}(B\Gamma) =(H_l(M; \bfL(\Z)) \to H_l(B\Gamma; \bfL(\Z))) \ar[d]^{\pr} \ar[u]_{\inc}
\\
E^{\infty}_{l,n}(M) \ar[r]^-{E^{\infty}_{l,n}(f)}  \ar[d]^{\id}_{\cong}
&
E^{\infty}_{l,n}(B\Gamma) \ar[d]^{\id}_{\cong}
\\
E^{2}_{l,n}(M) \ar[r]^-{E^{2}_{l,n}(f)} \ar[d]^{\id}_{\cong}
&
 E^{2}_{l,n}(B\Gamma) \ar[d]^{\id}_{\cong}
\\
H_l^{\IZ/p}\bigl(S^l;H_n(T^n_{\rho};\bfL(\Z))\bigr) \ar[r]^-{g_{n+l}} \ar[d]^{\id}_{\cong}
&
 H_l^{\IZ/p}\bigl(E\IZ/p; H_n(T^n_{\rho};\bfL(\Z))\bigr)
\ar[d]^{\id}_{\cong}
\\
H_n(T^n_{\rho};\bfL(\Z))^{\IZ/p} \ar[r]^-{g_{n+l}} \ar[d]^{A_n(T^n_{\rho})^{\IZ/p}}_{\cong}
&
 H_l\bigl(\IZ/p;H_n(T^n_{\rho};\bfL(\Z))\bigr) 
\ar[d]^{\id}_{\cong}
\\
L_n^s(\IZ[\IZ^n_{\rho}])^{\IZ/p} 
\ar[r]^-{g_{n+l} \circ (A_n(T^n_{\rho})^{\IZ/p})^{-1}}
&
H_l\bigl(\IZ/p;H_n(T^n_{\rho};\bfL(\Z))\bigr) 
}
& & 
\end{eqnarray}

Some explanations are in order.  Here, and elsewhere in the paper, if $A$ is a nontrivial
$\Z G$-module and $X$ is a free $G$-space, we write $H^{G}_i(X;A)$ for
$H_i(C_*(X) \otimes_{\Z G} A)$, and we also write $H_i(G;A)$ for $H^G_i(EG;A)$.

The symbol $\inc$ stands always for an obvious inclusion.

Given a free $\IZ/p$-$CW$-complex $X$ and any homology theory $\calh_*$, there is 
a Leray-Serre spectral sequence converging to $\calh_{i+j}(X \times_{\IZ/p} T^n_{\rho})$
whose $E^2$-term is  $H^{\IZ/p}_i\bigl(X;\calh_j(T^n_{\rho})\bigr)$. In particular, we have a spectral sequence
\[
E^2_{i,j} = H_i^{\Z/p}(X;  H_j(T^n_\rho; \bfL(\Z))) \Longrightarrow H_{i+j}(X \times_{\Z/p} T^n_{\rho}; \bfL(\Z)).
\]
 The symbols $F_{l,n}(M)$, $F_{l,n}(B\Gamma)$, $E_{l,n}^r(M)$, and $E_{l,n}^r(B\Gamma)$ 
 denote the corresponding filtration terms and
$E^r$-terms of the spectral sequences applied to the free $\IZ/p$-$CW$-complex
$X = S^l$ and $X = E\IZ/p$. This explains
the third, fourth and fifth row in the diagram~\eqref{big_commutative_surgery_diagram}
except for the map $g_{n+l}$ which we describe below.
In order to show the equality of the fourth, fifth and the sixth row we need to show

\begin{lemma}\label{lem:vanishing_of_differentials_in_three_spectral_sequences}
All differentials in the following  two spectral sequences vanish:
\begin{enumerate}

\item\label{lem:vanishing_of_differentials_in_three_spectral_sequences:(1)}
$E^2_{i,j} = H_i\bigl(\IZ/p;H_j(T^n_{\rho};\bfL(\Z))\bigr) \Longrightarrow H_{i+j}(B\Gamma;\bfL(\Z))$;

\item\label{lem:vanishing_of_differentials_in_three_spectral_sequences:(2)}
$E^2_{i,j} = H_i^{\IZ/p}\bigl(S^l;H_j(T^n_{\rho};\bfL(\Z))\bigr) \Longrightarrow H_{i+j}(M;\bfL(\Z))$.

\end{enumerate}
\end{lemma}
\begin{proof}~%
\ref{lem:vanishing_of_differentials_in_three_spectral_sequences:(1)}
It suffices to show that all differentials vanish 
after inverting $p$ and after localizing at $p$.
Since for $i \not= 0$
\[
E^2_{i,j}[1/p] = H_i\bigl(\Z/p;H_j(T^n_{\rho};\bfL(\Z))\bigr)[1/p] = 0,
\]
this is obvious after inverting $p$.

If we localize at $p$, we get a natural isomorphism of  homology theories
\[
KO_*(-)_{(p)} \xrightarrow{\cong} H_*(-;\bfL(\Z))_{(p)},
\]
since $p$ is odd, by Sullivan's $KO[1/2]$-orientation, see~Theorem~\ref{the:Comparing_L-theory_and_topological_K-theory}.
Hence it suffices to show that all the
differentials of  the Leray-Serre spectral sequence converging to $KO_{i+j}(B\Gamma)$
with $E^2$-term $H_i\bigl(\IZ/p,KO_j(T^n_{\rho})\bigr)$ are trivial.  The edge
homomorphism
\[
H_0(\Z/p;KO_{m}(T^n_{\rho})) = KO_{m}(T^n_{\rho}) \otimes_{\IZ[\IZ/p]} \IZ 
\xrightarrow{\cong} KO_{m}(B\Gamma)
\]
is bijective for even $m$ by~\cite[Theorem 6.1~(ii)]{Davis-Lueck(2013)}.
Thus all differentials involving $E^r_{0,m}$ are trivial for $m$ even.  Hence it suffices
to show
\[
H_i\bigl(\Z/p;KO_j(T^n_{\rho})\bigr)_{(p)} = 0 \quad 
\text{if}\; i > 0, \; i + j \equiv  0 \mod 2.
\]
Since there are natural transformation of homology theories
$i \colon KO_* \to K_*$ 
and $r \colon K_* \to KO_*$ with
$r \circ i =  2 \cdot \id$, it suffices to show
\[
H_i\bigl(\Z/p;K_j(T^n_{\rho})\bigr)_{(p)} = 0 \quad 
\text{if}\; i > 0, \; i + j \equiv  0 \mod 2.
\]
This vanishing is explicitly given in the proof of~\cite[Theorem 4.1~(ii)]{Davis-Lueck(2013)}.
\\[1mm]~%
\ref{lem:vanishing_of_differentials_in_three_spectral_sequences:(2)}
Let $g \colon S^l \to E\IZ/p$ be the classifying map of the free $\IZ/p$-CW-complex $S^l$.
It induces a map of the spectral sequence of 
assertion~\ref{lem:vanishing_of_differentials_in_three_spectral_sequences:(2)}
to the one of assertion~\ref{lem:vanishing_of_differentials_in_three_spectral_sequences:(1)}.
We know already the all differentials of the latter one vanish.
The induced maps on the $E^2$-terms
\begin{equation}
E^2_{i,j}(M) = H^{\IZ/p}_i\bigl(S^l;H_j(T^n_{\rho};\bfL(\Z))\bigr) 
\to E^2_{i,j}(B\Gamma) = H^{\IZ/p}_i\bigl(E\IZ/p;H_j(T^n_{\rho};\bfL(\Z))\bigr)
\label{Gorillaz}
\end{equation}
are bijective for $i \le l-1$ and surjective for $ i = l$
since $S^l \to E\IZ/p$ is $l$-connected. Since $S^l$
is $l$-dimensional, we have
\[
E^2_{i,j}(M) = 0 \quad \text{if}\; i \ge l+1.
\]
This finishes the proof of Lemma~\ref{lem:vanishing_of_differentials_in_three_spectral_sequences}.
\end{proof}

Let $g_{n+l} \colon H^{\Z/p}_l\bigl(S^l;H_n(T_{\rho}^n;\bfL(\Z))\bigr) \to H^{\Z/p}_l\bigl(E\IZ/p;H_n(T^n_{\rho};\bfL(\Z))\bigr)$
be the map induced by the classifying map $g \colon S^l \to E\IZ/p$.

For an appropriate choice of generator $t \in \IZ/p$ 
the chain $\IZ[\IZ/p]$-chain complex of $S^l$ is $\IZ[\IZ/p]$-chain homotopy equivalent
to the $l$-dimensional $\IZ[\IZ/p]$-chain complex
\[
\cdots \to 0 \to \IZ[\IZ/p] \xrightarrow{t-1}  \IZ[\IZ/p] \xrightarrow{N}
 \cdots \xrightarrow{N}  \IZ[\IZ/p] \xrightarrow{t-1}  \IZ[\IZ/p] \to 0 \to \cdots.
\]
Thus we obtain an  identification 
\[
H_l^{\IZ/p}\bigl(S^l;H_n(T^n_{\rho};\bfL(\Z))\bigr) =
H_n(T^n_{\rho};\bfL(\Z))^{\IZ/p}.
\]

The assembly map
\[
A_n(T^n_{\rho}) \colon H_n(T^n_{\rho};\bfL(\Z))  \xrightarrow{\cong} L_n(\IZ[\IZ^n_{\rho}])
\]
is an isomorphism because of the Shaneson splitting (or the Farrell-Jones conjecture in
$L$-theory for the group $\Z^n$) and the Rothenberg sequences since $\Wh(\IZ^n) = 0$ and
$\widetilde{K}_m(\IZ[\IZ^n]) = 0$ for $m \le 0$, and is a map of $\Z[\Z/p]$-modules
because of naturality of the assembly map.

Define the map 
\begin{eqnarray}
\mu  \colon \strp_{n+l+1}(M) 
& \to &
L_n(\IZ[\IZ^n_{\rho}])^{\IZ/p}
\label{map_mu}
\end{eqnarray} 
to be the composite of the vertical arrows in
the diagram~\eqref{big_commutative_surgery_diagram} from 
$\strp_{n+l+1}(M)$ to $L_n^s(\IZ[\IZ^n_{\rho}])^{\IZ/p}$, namely, to be the composite
$A_n(T_{\rho}^n)^{\IZ/p} \circ \pr \circ  \eta_{n+l+1}(M)$.

We have explained all modules and maps in the 
diagram~\eqref{big_commutative_surgery_diagram}. One easily checks that it commutes.
Hence we have accomplished Goal 1).

\begin{lemma}\label{lem:comparing_periodic_structure_sets_of_M_and_BGamma}
Consider the following commutative square
\[
\xymatrix@!C=9em{
\strp_{n+l+1}(M) \ar[r]^{\strp_{n+l+1}(f)} \ar[d]^-{\eta_{n+l+1}(M)}
&
\strp_{n+l+1}(B\Gamma)   \ar[d]^-{\eta_{n+l+1}(B\Gamma)} 
\\
H_{n+l}(M;\bfL(\Z))  \ar[r]^{H_{n+l}(f;\bfL(\Z))}
&
H_{n+l}(B\Gamma;\bfL(\Z))
} 
\]
where the vertical maps come from~\eqref{algebraic_surgery_sequence}.

Then the  vertical maps induce an isomorphism
\begin{eqnarray*}
\ker\left(\strp_{n+l+1}(f)\right)
& \xrightarrow{\cong} &
\ker\left(H_{n+l}(f;\bfL(\Z))\right).
\end{eqnarray*}
\end{lemma}
\begin{proof}
Consider the following commutative diagram of abelian groups which comes from~\eqref{algebraic_surgery_sequence}
\[
\xymatrix@!C=13em{%
H_{n+l+1}(M;\bfL(\Z)) \ar[r]^{H_{n+l+1}(f;\bfL(\Z))} \ar[d]^{A_{n+l+1}(M)}
&
H_{n+l+1}(B\Gamma;\bfL(\Z)) \ar[d]^{A_{n+l+1}(B\Gamma)}
\\
L_{n+l+1}^s(\IZ\Gamma) \ar[r]^{\id} \ar[d]^{\xi_{n+l+1}(M)} 
&
L_{n+l+1}^s(\IZ\Gamma)  \ar[d]^{\xi_{n+l+1}(B\Gamma)} 
\\
\strp_{n+l+1}(M) \ar[r]^{\strp_{n+l+1}(f)} \ar[d]^-{\eta_{n+l+1}(M)}
&
\strp_{n+l+1}(B\Gamma)  \ar[d]^-{\eta_{n+l+1}(B\Gamma)}
\\
H_{n+l}(M;\bfL(\Z)) \ar[r]^{H_{n+l}(f;\bfL(\Z))} \ar[d]^{A_{n+l}(M)}
&
H_{n+l}(B\Gamma;\bfL(\Z)) \ar[d]^{A_{n+l}(B\Gamma)}
\\
L_{n+l}^s(\IZ\Gamma) \ar[r]^{\id} 
&
L_{n+l}^s(\IZ\Gamma).
}
\]
An easy diagram chase shows that it suffices to show that the map
\[
H_{n+l+1}(f;\bfL(\Z)) \colon H_{n+l+1}(M;\bfL(\Z)) \to H_{n+l+1}(B\Gamma;\bfL(\Z))
\]
is surjective. We prove surjectivity after localizing at $p$ and 
after inverting $p$.  

We begin with localizing at $p$.
Since $p$ is odd,  for every $CW$-complex $X$ and
$m \in \IZ$ there is a natural isomorphism, see Theorem~\ref{the:Comparing_L-theory_and_topological_K-theory},
\[
H_{m}(X;\bfL(\Z))_{(p)} \xrightarrow{\cong} KO_m(X)_{(p)}.
\]
Hence it suffices to show that
\[
KO_{n+l+1}(f)_{(p)} \colon KO_{n+l+1}(M)_{(p)} \to KO_{n+l+1}(B\Gamma)_{(p)}
\]
is surjective. This follows from the following commutative diagram
\[
\xymatrix{%
KO_{n+l+1}(S^l \times B\IZ^n) \ar[r] \ar[d]^{KO_{n+l+1}(\pr)}
&
KO_{n+l+1}(M) \ar[d]^{KO_{n+l+1}(f)}
\\
KO_{n+l+1}(B\IZ^n) \ar[r] 
&
KO_{n+l+1}(B\Gamma) 
}
\]
since the left vertical arrow is induced by the projection and is hence surjective and the
lower horizontal arrow is surjective, as $KO_{k}(B\IZ^n) \to KO_{k}(B\Gamma)$ is surjective 
for $k$ even by~\cite[Theorem 6.1~(ii)]{Davis-Lueck(2013)}.

Next we invert $p$.  Then a standard transfer argument (see, for example, Proposition A.4
of~\cite{Davis-Lueck(2013)}) shows that
\[
H_m(B\IZ^n_{\rho};\bfL(\Z)) \otimes_{\IZ[\IZ/p]} \IZ[1/p] 
\xrightarrow{\cong} 
H_m(B\Gamma;\bfL(\Z))[1/p].
\]
is an isomorphism for all $m \in \Z$.  It follows, as above, that
$H_{n+l+1}(M;\bfL(\Z)) \to H_{n+l+1}(B\Gamma;\bfL(\Z))$ is surjective, as desired.
\end{proof}

\begin{lemma}\label{lem:comparing:_kernel_onH_and_E}
The composite
\[
H_{n+l}(M;\bfL(\Z)) \xrightarrow{\id} F_{l,n}(M) \xrightarrow{\pr}
E^{\infty}_{l,n}(M)
\]
induces a bijection
\begin{multline*}
\ker\bigl(H_{n+l}(f;\bfL(\Z)) \colon H_{n+l}(M;\bfL(\Z)) 
\to  H_{n+l}(B\Gamma;\bfL(\Z))\bigr)
\\
\xrightarrow{\cong} 
\ker\bigl( E^{\infty}_{l,n}(f) \colon E^{\infty}_{l,n}(M) \to
  E^{\infty}_{l,n}(B\Gamma)\bigr).
\end{multline*}
\end{lemma}
\begin{proof}
We have shown in 
Lemma~\ref{lem:vanishing_of_differentials_in_three_spectral_sequences}~%
\ref{lem:vanishing_of_differentials_in_three_spectral_sequences:(1)}
and~\ref{lem:vanishing_of_differentials_in_three_spectral_sequences:(2)}
that all the differentials of the spectral sequence converging to $H_{i+j}(B\Gamma;\bfL(\Z))$ with 
$E^2_{i,j} = H_i\bigl(\IZ/p;H_j(T^n_{\rho};\bfL(\Z))\bigr)$ and all the differentials  of the
spectral sequence  converging to $H_{i+j}(M;\bfL(\Z))$ with 
$E^2_{i,j} = H_i^{\IZ/p}\bigl(S^l;H_j(T^n_{\rho};\bfL(\Z))\bigr)$
vanish. The map
\[
E^2_{i,j}(f) \colon E^2_{i,j}(M) \to E^2_{i,j}(B\Gamma)
\]
is bijective for $i \le l-1$ and all $j$ since the map $S^l \to E\IZ/p$ is
$l$-connected. Hence the map
\[
F_{i,j}(f) \colon F_{i,j}(M) \to F_{i,j}(B\Gamma)
\]
is bijective for $i \le l-1$ and all $j$. This implies
\[
F_{l-1,n+1}(M) \cap \ker\bigl(H_{n+l}(f;\bfL(\Z)) \colon H_{n+l}(M;\bfL(\Z)) 
\to  H_{n+l}(B\Gamma;\bfL(\Z))\bigr) = 0.
\]
Since $S^l$ is $l$-dimensional and hence $H_{n+l}(M;\bfL(\Z)) =
F_{l,n}(M)$, Lemma~\ref{lem:comparing:_kernel_onH_and_E} follows.
\end{proof}

Lemma~\ref{lem:comparing_periodic_structure_sets_of_M_and_BGamma} and
Lemma~\ref{lem:comparing:_kernel_onH_and_E} give Goal 2) and hence complete the proof of
 assertion~\ref{the:The_periodic_simple_structure_set_of_M:inj_map_strp} 
of Theorem~\ref{the:The_periodic_simple_structure_set_of_M}
\end{proof}

\begin{proof}[Proof of assertion~\ref{the:The_periodic_simple_structure_set_of_M:image}
of Theorem~\ref{the:The_periodic_simple_structure_set_of_M}]
Consider the commutative diagram
\begin{equation}\label{technical_commutative}
\xymatrix@!C=9em{%
 F_{l-1,n+l}(M) \ar[r]^{\cong} \ar[d]_{\inc} 
& 
F_{l-1,n+l}(B\Gamma) \ar[d]^{\cong_{1/p}} 
\\
H_{n+l}(M;\bfL(\Z)) \ar[d]_{A_{n+l}(M)} \ar[r]^{H_{n+l}(f;\bfL(\Z))}
& 
H_{n+l}(B\g; \bfL(\Z)) \ar[d]^{\cong_{1/p}} \ar[dl]^{A_{n+l}(B\g)}
\\
 L_{n+l}^s (\Z \g) \ar[r]_{\cong} 
& 
H_{n+l}(\bub \g; \bfL(\Z))
}
\end{equation}
where the two maps ending at $L_{n+l}^s (\Z \g)$ are the assembly maps $A_{n+l}(M)$ and
$A_{n+l}(B\g)$ and bottom horizontal isomorphism is discussed in 
Theorem~\ref{the:L(ZGamma)_decorated}~\ref{the:L(ZGamma)_decorated:odd}.
We have already explained that the top horizontal arrow is an isomorphism, see~\eqref{Gorillaz}.  
The spectral sequence appearing in
Lemma~\ref{lem:vanishing_of_differentials_in_three_spectral_sequences}~%
\ref{lem:vanishing_of_differentials_in_three_spectral_sequences:(1)}
implies that the vertical inclusion at the top right induces an isomorphism after
tensoring with $\Z[1/p]$. We conclude from Lemma~\ref{lem:calh(bub(Gamma))_to_calh(Zn)Z/p}
that $H_{n+l}(B\g; \bfL(\Z))[1/p] \to H_{n+l}(\bub\g; \bfL(\Z))[1/p]$ is an isomorphism.
 
Now the diagram~\eqref{technical_commutative} shows that
\[
\left(A_{n+l}(M) \circ \inc\right)[1/p] \colon  
F_{l-1,n+1}(M)[1/p] \xrightarrow{\cong} L_{n+l}^s(\IZ \Gamma)[1/p]
\]
is bijective.

We have the following commutative diagram
\begin{eqnarray}
\label{diagram_SHLsF}
& & \xymatrix@!C=11em{
& 
0 \ar[d] 
& 
\\ 
& 
F_{l-1,n+1}(M) \ar[d]^{\inc} \ar[dr]^{\hspace{3mm}A_{n+l}(M) \circ \inc} 
&
\\
\strp_{n+l+1}(M) \ar[r]^{\eta_{n+l+1}(M)} \ar[dr]^{\mu}
&
H_{n+l}(M;\bfL(\Z)) \ar[r]^{A_{n+l}(M)} \ar[d]^{A_{n+l}(T^n_{\rho})^{\IZ/p} \circ \pr}
&
L_{n+l}^s(\IZ \Gamma)
\\
&
L_n(\IZ[\IZ^n_{\rho}])^{\IZ/p} \ar[d]
&
\\
& 
0
&
}
\end{eqnarray}
 with exact row and exact column. An easy diagram chase proves that
\[
\mu[1/p] \colon \strp_{n+l+1}(M)[1/p] \to L_n(\IZ[\IZ^n_{\rho}])^{\IZ/p}[1/p]
\]
is surjective and we get
\begin{equation}
\label{equality_of_certain_kernels}
\ker(\mu[1/p]) = \ker(\eta_{n+l+1}(M)[1/p]).
\end{equation}
Since $L_n^s(\IZ[\IZ^n_{\rho}])$ is a finitely generated abelian group, see 
Theorem~\ref{the:The_periodic_structure_set_of_BP_for_a_finite_p-group}~%
\ref{the:The_periodic_structure_set_of_BP_for_a_finite_p-group:point},
the cokernel of $\mu$ is a finite abelian $p$-group.
Recalling that the cokernel of $\sigma$ is isomorphic to the cokernel of $\mu$, his finishes the proof  of assertion~\ref{the:The_periodic_simple_structure_set_of_M:image}
of Theorem~\ref{the:The_periodic_simple_structure_set_of_M}
\end{proof}
\begin{proof}[Proof of assertion~\ref{the:The_periodic_simple_structure_set_of_M:kernel}
of Theorem~\ref{the:The_periodic_simple_structure_set_of_M}]
The composite of 
\[
\nu \colon \bigoplus_{(P) \in \calp} \widetilde{L}^s_{n+l+1}(\IZ P) \to \strp_{n+l+1}(M)
\]
with
\[
\strp_{n+l+1}(f) \colon \strp_{n+l+1}(M) \to \strp_{n+l+1}(B\Gamma)
\]
becomes an isomorphism after inverting $p$ since the composite is also the composite of
\[
\bigoplus_{(P) \in \calp} \widetilde{L}^s_{n+l+1}(\IZ P) \to \bigoplus_{(P) \in \calp} \strp_{n+l+1}(BP) \to \strp_{n+l+1}(B\g),
\]
where the first map is an isomorphism after inverting $p$ by 
Theorem~\ref{the:The_periodic_structure_set_of_BP_for_a_finite_p-group}~%
\ref{the:The_periodic_structure_set_of_BP_for_a_finite_p-group:localization}
and the second map is an isomorphism by 
Theorem~\ref{the:Computation_of_strp(BGamma)}~%
\ref{the:Computation_of_strp(BGamma):induction_iso}.
Hence $\nu$ is injective after inverting $p$.  Since
$\widetilde{L}^s_{n+l+1}(\IZ P )$ is a finitely generated free abelian group, $\nu$ is
injective.

Note
\[\im(\nu) \subset \im(\wt \xi_{n+l+1}(M)) = \im(\xi_{n+l+1}(M)) =
\ker(\eta_{n+l+1}(M))\subset \ker \mu,
\]
where the first equality holds since the
simply-connected surgery exact sequence is short exact and the second equality holds
because of the periodic surgery exact sequence~\eqref{algebraic_surgery_sequence}.
We have  shown in~\eqref{equality_of_certain_kernels}
that $\ker(\mu[1/p]) = \ker(\eta_{n+l+1}(M)[1/p])$ holds. Hence
$\ker(\mu)/\ker \eta_{n+l+1}(M) = \ker(\mu)/\im(\xi_{n+l+1}(M))$ is a $p$-torsion abelian
group.  Since $\strp_{n+l+1}(M)$ is a finitely generated abelian group because of the
surgery exact sequence~\eqref{algebraic_surgery_sequence}, we conclude that
$\ker(\mu)/\im(\xi_{n+l+1}(M))$ is a finite abelian $p$-group.

Next we consider the following commutative diagram
\[
\xymatrix@!C=10em{
&
0 \ar[d]
&
\\
& 
\bigoplus_{(P) \in \calp} \widetilde{L}^s_{n+l+1}(\IZ P) \ar[d] \ar[dr]^{\nu}
&
\\
H_{n+l+1}(B\Gamma;\bfL(\Z)) \ar[r]^{A_{n+l+1}(B\Gamma)} \ar[dr]_{H_{n+l+1}(i;\bfL(\Z))}
& 
L^s_{n+l+1}(\IZ \Gamma) \ar[r]^{\xi_{n+l+1}(B\Gamma)} \ar[d]
& \strp_{n+l+1}(M)
\\
& H_{n+l+1}(\bub{\Gamma};\bfL(\Z))  \ar[d]
&
\\
& 
0
&
}
\]
Here the exact column is taken from 
Theorem~\ref{the:L(ZGamma)_decorated}~\ref{the:L(ZGamma)_decorated:exact_sequence}, 
the row is exact because of the 
surgery exact sequence~\eqref{algebraic_surgery_sequence} for $M$ and the surjectivity
of $H_{n+l+1}(M;\bfL(\Z)) \to H_{n+l+1}(B\Gamma;\bfL(\Z))$
which we have shown in the proof of 
Lemma~\ref{lem:comparing_periodic_structure_sets_of_M_and_BGamma}.
The map $H_{n+l+1}(i;\bfL(\Z))$ induced by the obvious map $i \colon B\Gamma \to
\bub{\Gamma}$ is bijective after inverting $p$, see
the proof of part~\ref{the:The_periodic_simple_structure_set_of_M}~%
\ref{the:The_periodic_simple_structure_set_of_M:image}.
Since there is a finite $CW$-model for $\bub{\Gamma}$ and hence
$H_{n+l+1}(\bub{\Gamma};\bfL(\Z))$ is finitely generated, the cokernel of
$H_{n+l+1}(i;\bfL(\Z))$ is a finite abelian $p$-group.
Since this cokernel is isomorphic to 
$\im(\xi_{n+l+1}(B\Gamma))/\im(\nu)$ by the commutative diagram above, we have shown that
both $\ker(\mu)/\im(\xi_{n+l+1}(M))$ and $\im(\xi_{n+l+1}(B\Gamma))/\im(\nu)$
are finite abelian $p$-groups. Therefore $\ker(\mu)/\im(\nu)$ is a finite abelian
$p$-group.  Recalling that $\ker(\mu)= \ker (\sigma)$, this finishes the proof  of
assertion~\ref{the:The_periodic_simple_structure_set_of_M:kernel}
of Theorem~\ref{the:The_periodic_simple_structure_set_of_M}.
\end{proof}

\begin{proof}[Proof of assertion~\ref{the:The_periodic_simple_structure_set_of_M:inverting_p}
of Theorem~\ref{the:The_periodic_simple_structure_set_of_M}]
Because of assertion~\ref{the:The_periodic_simple_structure_set_of_M:inj_map_strp}
it suffices to show that
\[
\left(\mu \times \strp_{n+l+1}(f)\right)[1/p] \colon \strp_{n+l+1}(M)[1/p] \to 
L_n(\IZ[\IZ^n_{\rho}])^{\IZ/p}[1/p]  \times \strp_{n+l+1}(B\Gamma)[1/p]
\]
is surjective. We have already shown that $\mu[1/p]$ is surjective and that
its kernel agrees with the image of $\nu[1/p]$. We have already explained
that $\strp_{n+l+1}(f)[1/p] \circ \nu[1/p]$ is an isomorphism.
Assertion~\ref{the:The_periodic_simple_structure_set_of_M:inverting_p} 
follows. 
\end{proof}

\begin{proof}[Proof of assertion~\ref{the:The_periodic_simple_structure_set_of_M:comp}]
If we invert $p$, we conclude from
assertion~\ref{the:The_periodic_simple_structure_set_of_M:inverting_p} 
using Theorem~\ref{the:Computation_of_strp(BGamma)}~%
\ref{the:Computation_of_strp(BGamma):induction_iso}, 
Theorem~\ref{the:The_periodic_structure_set_of_BP_for_a_finite_p-group}, 
Lemma~\ref{lem:preliminaries_about_Gamma_and_Zn_rho}~%
\ref{lem:preliminaries_about_Gamma_and_Zn_rho:order_of_calp},
\begin{eqnarray*}
\strp_{n+l+1}(M)[1/p] &\cong & 
\IZ^{p^k(p-1)/2}[1/p] \oplus L_{n}(\IZ[\IZ^n_{\rho}])^{\IZ/p}[1/p].
\end{eqnarray*}
The abelian groups 
$\strp_{n+l+1}(M)$ and $\IZ^{p^k(p-1)/2} \oplus L_{n}(\IZ[\IZ^n_{\rho}])^{\IZ/p}$
are finitely generated. The abelian group
$\IZ^{p^k(p-1)/2} \oplus L_{n}(\IZ[\IZ^n_{\rho}])^{\IZ/p}$ contains no $p$-torsion.
We conclude from assertion~\ref{the:The_periodic_simple_structure_set_of_M:inj_map_strp}
that the abelian group $\strp_{n+l+1}(M)$ contains no $p$-torsion.
Hence 
\begin{eqnarray*}
\strp_{n+l+1}(M) &\cong & 
\IZ^{p^k(p-1)/2} \oplus L_{n}(\IZ[\IZ^n_{\rho}])^{\IZ/p}.
\end{eqnarray*}
It remains to prove
\begin{eqnarray*}
L_{n+l}(\IZ[\IZ^n_{\rho}])^{\IZ/p} & \cong &
\bigoplus_{i = 0}^n L_{n-i}(\IZ)^{r_i}.
\end{eqnarray*}
This follows from Remark~\ref{remark:points_of_view_sigma}.  
This finishes the proof of
Theorem~\ref{the:The_periodic_simple_structure_set_of_M}.
\end{proof}


\typeout{-------------- The geometric simple structure set of M  -------------------------}


\section{The geometric simple structure set of $M$}
\label{sec:The_geometric_simple_structure_set_of_M}

In this section we compute the geometric simple structure set
$\strg(M)$ of $M$.  For the rest of this paper, we simplify our notation and write
$\bfL$ for the spectrum $\bfL(\Z)$ and $\bfL\langle 1 \rangle$ for its
1-connective cover.

In general we have the following relationship between the periodic and
the geometric simple structure set.

\begin{lemma}
\label{lem:strg_strp_injective}
Let $N$ be a connected oriented closed  $m$-dimensional manifold. Then we obtain an exact sequence
\[
0 \to \strg(N) \xrightarrow{j(N)} \strp_{m+1}(N)  \xrightarrow{\partial(N)} H_m(N;\bfL/\bfL\langle 1 \rangle)
\]
where the map $j(N)$ is taken from~\eqref{j(N)}, the map $\partial(N)$ factors as
$\strp_{m+1}(N) \xrightarrow{\eta_{m+1}(N)} H_m(N;\bfL) \to H_m(N;\bfL/\bfL\langle 1\rangle)$, and
we have $H_m(N;\bfL/\bfL\langle 1\rangle) \cong H_m(N;L_0(\Z)) \cong \Z$.
\end{lemma}
\begin{proof}
We have the following commutative diagram with exact columns
\[
\xymatrix@!C=13em{
H_{m+1}(N;\bfL\langle 1 \rangle) \ar@{->>}[r] \ar[d]
&
H_{m+1}(N;\bfL) \ar[d]
\\
L_{m+1}^s(\IZ\Gamma) \ar[r]^{\id} \ar[d] 
&
L_{m+1}^s(\IZ\Gamma)  \ar[d]
\\
\strg(N) \ar[r]^{j(N)} \ar[d]
&
\strp_{m+1}(N)  \ar[d]
\\
H_{m}(N;\bfL\langle 1 \rangle ) \ar@{>->}[r] \ar[d]
&
H_{m}(N;\bfL) \ar[d]
\\
L_{m}^s(\IZ\Gamma) \ar[r]^{\id}
&
L_{m}^s(\IZ\Gamma),
}
\]
where the columns are the exact sequences~\eqref{algebraic_surgery_sequence}
and~\eqref{algebraic_surgery_sequence_one_connective}  using the identification~\eqref{ident_structure_sets}
and the horizontal maps are given by the passage 
$\bfL\langle 1 \rangle \to \bfL$. Let $\bfL/\bfL\langle 1 \rangle$ be the homotopy cofiber of
the canonical map $\bfi \colon \bfL\langle 1 \rangle \to \bfL$ and
denote by $\bfpr \colon \bfL \to \bfL/\bfL\langle 1 \rangle$
the canonical map of spectra. We get an exact sequence
\begin{multline*}
H_{m+1}(N;\bfL\langle 1 \rangle) \to H_{m+1}(N;\bfL)  
\to H_{m+1}(N;\bfL/\bfL\langle 1 \rangle) 
\\
\to H_{m}(N;\bfL\langle 1 \rangle) \to H_{m}(N;\bfL) \to H_{m}(N;\bfL/\bfL\langle 1 \rangle).
\end{multline*}
Since $\pi_q(\bfL/\bfL\langle 1 \rangle)$ vanishes for $q \ge 1$ and is $L_0(\IZ)$ for
$q = 0$, an easy spectral sequence argument shows that
$H_{m+1}(N;\bfL/\bfL\langle 1 \rangle) = 0$.  Thus the top horizontal map is surjective.
The fourth horizontal map is injective and its cokernel maps injectively to
$H_{m}(N;\bfL/\bfL\langle 1 \rangle) \cong H_m(N;L_0(\IZ)) \cong \Z$.  A diagram chase
yields the desired exact sequence. 
\end{proof}

Recall that $M = T^n_{\rho} \times_{\Z/p} S^l$ and $\Gamma = \Z^n \times_{\rho} \Z/p$
is $\pi_1(M)$.  Let $f \colon M \to B\Gamma$ be a classifying map for the universal
covering of $M$.

\begin{theorem}[The geometric simple structure set of $M$]\
\label{the:The_geometric_simple_structure_set_of_M}
There is a homomorphism
\[
\sigma^{\geo} \colon \strg(M) \to H_n(T^n;\bfL\langle 1 \rangle)^{\IZ/p},
\]
such that the following holds:  
\begin{enumerate}
\item\label{the:The_geometric_simple_structure_set_of_M:inj_map_strp}
The map
\[
\sigma^{\geo}  \times (\strp_{n+l+1}(f) \circ j(M))\colon \strg(M) \to 
H_n(T^n;\bfL\langle 1 \rangle)^{\IZ/p} \times \strp_{n+l+1}(B\Gamma)
\] 
is injective;

\item\label{the:The_geometric_simple_structure_set_of_M:image}
The cokernel of $\sigma^{\geo} $ is a finite abelian $p$-group;

\item\label{the:The_geometric_simple_structure_set_of_M:kernel}
Consider the composite 
\[
\nu^{\geo} \colon \bigoplus_{(P) \in \calp} \widetilde{L}^s_{n+l+1}(\IZ P) \to 
\widetilde{L}^s_{n+l+1}(\IZ \Gamma)
\xrightarrow{\widetilde{\xi}^{\langle 1 \rangle}_{n+l+1}(M)} \strg(M)
\]
where the first map is given by induction with the various inclusions $P \to \Gamma$
and $\widetilde{\xi}^{\langle 1 \rangle}_{n+l+1}(M)$ comes from~\eqref{reduced_surgery_sequence_geometric}.

Then $\nu^{\geo}$ is injective, the image of $\nu^{\geo}$ is contained in the kernel of $\sigma^{\geo}$ and
$\ker(\sigma^{\geo})/\im(\nu^{\geo})$ is a finite abelian $p$-group;

\item\label{the:The_geometric_simple_structure_set_of_M:inverting_p}
After inverting $p$ we obtain an isomorphism
\begin{multline*}
\hspace{20mm}  \left(\sigma^{\geo} \times (\strp_{n+l+1}(f) \circ j(M))\right)[1/p] \colon \strg(M)[1/p] 
\\
\to 
H_n(T^n;\bfL\langle 1 \rangle)^{\IZ/p}[1/p]  \times \strp_{n+l+1}(B\Gamma)[1/p];
\end{multline*}

\item\label{the:The_geometric_simple_structure_set_of_M:comp}
As an abelian group we have
\begin{eqnarray*}
\strg(M) &\cong & 
\IZ^{p^k(p-1)/2} \oplus \bigoplus_{i = 0}^{n-1} L_{n-i}(\IZ)^{r_i},
\end{eqnarray*}
where the numbers $r_i$ are defined in~\eqref{r-numbers};

\item\label{the:The_geometric_simple_structure_set_of_M:partial(M)}
The cokernel of the map
$\partial(M) \colon \strp_{d+l+1}(M) \xrightarrow{\eta_{d+l +1}(M)} H_{d+l}(M;\bfL/\bfL\langle 1\rangle)$
appearing in Lemma~\ref{lem:strg_strp_injective} is a finite cyclic $p$-group.

\end{enumerate}
\end{theorem}

\begin{remark}\label{remark:points_of_view_sigma_geo}
 There are several different points of views on the codomain of $\sigma^{\geo}$ (see Remark~\ref{remark:points_of_view_sigma}). Indeed there are isomorphisms
 \[
 H_n(T^n;\bfL\langle 1 \rangle )^{\Z/p} \cong \bigoplus_{i = 0}^{n-1} (H_i(T^n)^{\Z/p} \otimes L_{n-i}(\Z))  \cong \bigoplus_{i = 0}^{n-1} L_{n-i}(\Z)^{r_i}.
 \]
\end{remark}

\begin{proof}We first prove~\ref{the:The_geometric_simple_structure_set_of_M:inverting_p}.
We construct a commutative diagram whose columns are exact after inverting $p$.
\begin{equation}
\xymatrix@!C=17em{%
0 \ar[d] 
&
0 \ar[d] 
\\
\strg(M) \ar[d]
\ar[r]^-{\sigma^{\geo} \times (\strp_{n+l+1}(f) \circ j(M))}
&
H_n(T^n_{\rho};\bfL\langle 1 \rangle)^{\IZ/p}  \times \strp_{n+l+1}(B\g)\ar[d]
\\
\strp_{n+l+1}(M) \ar[d]
\ar[r]^-{\sigma \times \strp_{n+l+1}(f)}_-{\cong_{1/p}}
&
H_n(T^n_{\rho};\bfL)^{\IZ/p} \times \strp_{n+l+1}(B\g) \ar[d]
\\
H_{n+l}(M;\bfL/\bfL\langle 1 \rangle) \ar[r]^-{\alpha}_-{\cong_{1/p}}
&
H_n(T^n_{\rho};\bfL/\bfL\langle 1 \rangle)^{\IZ/p}.
}
\label{Tolisso}
\end{equation}
The exact left column is due to Lemma~\ref{lem:strg_strp_injective}.  In the right column,
the first nontrivial map is induced by the product of the change of coefficients map
$\bfL\langle 1 \rangle \to \bfL$ with the identity on the structure group and the second
nontrivial map is induced by the composite of projection on the torus factor and the
change of coefficients map $\bfL \to \bfL/\bfL\langle 1 \rangle$.  The right column is
exact after inverting $p$, since $H_{n+1}(T^n_{\rho};\bfL/\bfL\langle 1 \rangle)= 0$; thus
we have the exact sequence of $\IZ[\IZ/p]$-modules
\[
0 \to H_n(T^n_{\rho};\bfL\langle 1 \rangle) \to H_n(T^n_{\rho};\bfL)
\to H_n(T^n_{\rho};\bfL/\bfL\langle 1 \rangle).
\]

Recall that all differentials of the spectral sequence 
\[
  E^2_{i,j} = H_i^{\IZ/p}\bigl(S^l;H_j(T^n_{\rho};\bfL)\bigr) \Longrightarrow H_{i+j}(M;\bfL)
\]
vanish by Lemma~\ref{lem:vanishing_of_differentials_in_three_spectral_sequences}~%
\ref{lem:vanishing_of_differentials_in_three_spectral_sequences:(2)}. This implies that 
for the spectral sequence
\[
E^2_{i,j} = H_i^{\IZ/p}\bigl(S^l;H_j(T^n_{\rho};\bfL\langle 1 \rangle )\bigr) 
\Longrightarrow H_{i+j}(M;\bfL\langle 1 \rangle)
  \]
all differentials which end or start at place $(s,t)$ vanish, provided that $s + t = n+l$. 
Hence we can define  analogously to $\sigma$ a map
\begin{equation}
\sigma^{\geo} \colon \strg(M)  \to H_n(T^n_{\rho};\bfL \langle 1 \rangle)^{\IZ/p}
\label{overline(mu)_upper_geo}
\end{equation}
such that the following diagram commutes
\[
\xymatrix@!C=11em{
\strg(M)  \ar[r]^-{\sigma^{\geo}} \ar[d]_{j(M)}
&
H_n(T^n_{\rho};\bfL \langle 1 \rangle)^{\IZ/p} 
\ar[d]^-{H_n(T^n_{\rho};\bfi)^{\IZ/p} }
\\
\strp_{n+l+1}(M) \ar[r]_-{\sigma}
& 
H_n(T^n_{\rho};\bfL )^{\IZ/p}.
}
\]
The homomorphism  $\alpha$ in diagram~\eqref{Tolisso} is given by the edge isomorphism
\[
H_l^{\IZ/p}\bigl(S^l;H_n(T^n_{\rho};\bfL/\bfL\langle 1 \rangle )\bigr)  \xrightarrow{\cong}
H_{n+l}(M;\bfL/\bfL\langle 1 \rangle)
\]
at $i = l$ and $j = n$ of the spectral sequence 
\[
E^2_{i,j} = H_i^{\IZ/p}\bigl(S^l;H_j(T^n_{\rho};\bfL/\bfL\langle 1 \rangle)\bigr) 
\Longrightarrow H_{i+j}(M;\bfL/\bfL\langle 1 \rangle),
\]
the canonical map 
\[
H_l^{\IZ/p}\bigl(S^l;H_n(T^n_{\rho};\bfL/\bfL\langle 1 \rangle)\bigr) 
\to H_n(T^n_{\rho};\bfL/\bfL\langle 1 \rangle)^{\IZ/p},
\]
which is an isomorphism after inverting $p$ since $l$ is odd,
the isomorphism of $\IZ[\IZ/p]$-modules
\[
H_n(T^n_{\rho};\bfL/\bfL\langle 1 \rangle) \xrightarrow{\cong} H_n(T^n_{\rho};L_0(\IZ)),
\]
which is the obvious edge homomorphism in the spectral sequence
\[
E^2_{i,j} =H_i(T^n;\pi_j(\bfL/\bfL\langle 1 \rangle))
\Longrightarrow H_{i+j}(T^n;\bfL/\bfL\langle 1 \rangle),
\]
and the bijectivity of the inclusion
\[
H_n(T^n_{\rho};L_0(\IZ))^{\IZ/p} \xrightarrow{\cong} H_n(T^n_{\rho};L_0(\IZ)).
\]
The second horizontal arrow $\sigma \times \strp_{n+l+1}(f)$ in diagram~\eqref{Tolisso} is an isomorphism after 
inverting $p$ by Theorem~\ref{the:The_periodic_simple_structure_set_of_M}~%
\ref{the:The_periodic_simple_structure_set_of_M:inverting_p}.
We leave it to the reader to check that the diagram~\eqref{Tolisso} commutes.
By the Five Lemma the upper horizontal arrow
\[\sigma^{\geo} \times (\strp_{n+l+1}(f) \circ j(M)) \colon \strg(M)  \to
H_n(T^n_{\rho};\bfL\langle 1 \rangle)^{\IZ/p}  \times \strp_{n+l+1}(M)
\]
is an isomorphism after inverting $p$. This finishes the proof of
assertion~\ref{the:The_geometric_simple_structure_set_of_M:inverting_p}.
\\[1mm]~\ref{the:The_geometric_simple_structure_set_of_M:inj_map_strp}
This follows from assertion~\ref{the:The_geometric_simple_structure_set_of_M:inverting_p}
since $\strg(M)$ is a subgroup of $\strp_{n+l+1}(M)$ and $\strp_{n+l+1}(M)$ contains no $p$-torsion by
Theorem~\ref{the:The_periodic_simple_structure_set_of_M}~%
\ref{the:The_periodic_simple_structure_set_of_M:comp}.
\\[1mm]~\ref{the:The_geometric_simple_structure_set_of_M:image}
This follows from assertion~\ref{the:The_geometric_simple_structure_set_of_M:inverting_p}
since $H_n(T^n_{\rho};\bfL \langle 1 \rangle)^{\IZ/p}$ is a finitely generated abelian group.
\\[1mm]~\ref{the:The_geometric_simple_structure_set_of_M:kernel}
We get from the diagram~\eqref{Tolisso}  the following commutative diagram with exact columns
\[
  \xymatrix{0 \ar[d] \ar[r]
&
0 \ar[d] \ar[r]
& 0 \ar[d]
\\
\bigoplus_{(P) \in \calp} \widetilde{L}^s_{n+l+1}(\IZ P) \ar[r]^-{\nu^{\geo}} \ar[d]^-{\id} 
&
\strg(M) \ar[r]^-{\sigma^{\geo}} \ar[d]^{j(M)}
&
H_n(T^n_{\rho};\bfL\langle 1 \rangle)^{\IZ/p} \ar[d]^{H_n(T^n_{\rho};\bfi)^{\IZ/p}} 
\\
\bigoplus_{(P) \in \calp} \widetilde{L}^s_{n+l+1}(\IZ P)  \ar[r]^-{\nu} \ar[d]
& 
\strp_{n+l+1}(M) \ar[r]^{\sigma} \ar[d]
&
H_n(T^n_{\rho};\bfL)^{\IZ/p} \ar[d]
\\
0 \ar[r]
&
H_{n+l}(M;\bfL/\bfL\langle 1 \rangle) \ar[r]^-{\alpha}
&
H_n(T^n_{\rho};\bfL/\bfL\langle 1 \rangle)^{\IZ/p}.
}
\]
We have already shown that the map $\alpha$ is an isomorphism after inverting $p$ and its
source is an infinite cyclic group.  Hence $\alpha$ is
injective. Theorem~\ref{the:The_periodic_simple_structure_set_of_M}~%
\ref{the:The_periodic_simple_structure_set_of_M:kernel} shows that $\nu$ is injective,  the kernel of $\sigma$ contains $\im(\nu)$ and
$\ker(\sigma)/\im(\nu)$ is a finite abelian $p$-group. An easy diagram chase shows
that $\sigma^{\geo}$ is injective, the image of $\sigma^{\geo}$ is contained in the kernel of
$\sigma^{\geo}$ and there is obvious isomorphism
\[\
\ker(\sigma^{\geo})/\im(\nu^{\geo}) \xrightarrow{\cong}
\ker(\sigma)/\im(\nu).
\]%
~\ref{the:The_geometric_simple_structure_set_of_M:comp} This follows from
Theorem~\ref{the:The_periodic_simple_structure_set_of_M}~%
\ref{the:The_periodic_simple_structure_set_of_M:comp} and the diagram~\eqref{Tolisso} by
inspecting the definition of the various maps and the identification
$H_{n+l}(M;L_0(\IZ)) \cong L_0(\IZ)$.
\\[1mm]~\ref{the:The_geometric_simple_structure_set_of_M:partial(M)}
This follows from the diagram~\eqref{Tolisso},
after we have shown that the 
the map $H_d(T^n;\bfL) \to H_d(T^n;\bfL/\bfL\langle 1\rangle)$ is  surjective.
The latter claim follows from the fact $T^n$ is stably homotopy equivalent to a wedge of spheres.
This finishes the proof of Theorem~\ref{the:The_geometric_simple_structure_set_of_M}.
\end{proof}


\typeout{-------------- Invariants for detecting the structure set of $M$  -------------------------}



\section{Invariants for detecting the structure set of $M$}
\label{subsec:Invariance_for_Detecting_the_structure_set_of_M}

Let $G$ be a finite group; let $R(G)$ be its complex representation ring.  Let
$\wt R(G) = R(G)/\langle\operatorname{reg}\rangle$ be the reduced regular representation
ring.  (Additively $R(G)$ and $\wt R(G)$ are $K_0(\C G)$ and $\wt K_0(\C G)$,
respectively.)  For $\varepsilon = \pm 1$, let $R(G)^\varepsilon$ and
$\wt R(G)^\varepsilon$ be the subgroups invariant under $V \mapsto \varepsilon \ol V$.
For example, $\wt R(\Z/p) \cong \Z[e^{2\pi i/p}]$ and both $\wt R(\Z/p)^{+1}$ and
$\wt R(\Z/p)^{-1}$ are free abelian of rank $(p-1)/2$ for $p$ an odd prime.

Let $V$ be a finitely generated $\R G$-module and let 
$$
B : V \times V \to \R
$$
be a $G$-equivariant $\varepsilon$-symmetric form.  Atiyah-Singer~\cite[Section
6]{Atiyah-Singer(1968b)} define the \emph{$G$-signature}
$\sign_G (V,B) \in R(G)^\varepsilon$.  If $W$ is a compact $2d$-dimensional oriented
manifold with a $G$-action, define its $G$-signature $\sign_G(W)$ to be the $G$-signature
of its intersection form.

A representing element of $L_{2d}^s(\Z G)$ defines a $G$-equivariant,
$(-1)^d$-symmetric form by first tensoring with $\R$ and then composing with the trace map
$\R G \to \R$, $\sum a_g g \mapsto a_e$.  This defines the \emph{multisignature} maps
\begin{align}\label{multi_signature-widetilde}
{\sign}_{G} & \colon  {L}^s_{2d}(\IZ G)\to
R(G)^{(-1)^d} \nonumber \\
\widetilde{\sign}_{G} & \colon  \widetilde{L}^s_{2d}( \IZ G)\to
\wt R(G)^{(-1)^d}
\end{align}

\begin{lemma}\label{lem:multisignature_iso_after_inverting_2}
  The multisignature homomorphisms $\sign_G$ and $\wt \sign_G$ are  $\Z[1/2]$-isomorphisms.  
\end{lemma}

\begin{proof}
See Wall~\cite{Wall(1999)}, Theorems 13A.3 and 13A.4 or Ranicki~\cite{Ranicki(1992)}, Propositions 22,14 and 22.34.  
\end{proof}

When $G$ is cyclic and odd order, according to~\cite[Theorem 13A.4]{Wall(1999)}, \\$\wt
\sign_G \colon \wt {L}^s_{2d}(\IZ G)\to 4 \wt
R(G)^{(-1)^d}$ is an isomorphism.

\begin{theorem}[Geometric structure set of $M$]
\label{the:geometric_structure_set_of_M}
Let $d = (n+l +1)/2$.   There are injective maps
\[
\sigma \times \rho \colon \strp_{n+l+1}(M) \to H_n(T^n; \bfL)^{\Z/p} \oplus \left(\bigoplus_{(P) \in \calp}
 \wt R(P')^{(-1)^d}[1/p]\right),
\]
and
\[
\sigma^{\geo} \times \rho^{\geo} \colon \strg(M) \to H_n(T^n; \bfL\langle 1 \rangle)^{\Z/p} \oplus \left(\bigoplus_{(P) \in \calp}
 \wt R(P')^{(-1)^d}[1/p]\right).
\]
The cokernels of these maps are trivial after tensoring with $\Z[1/2p]$.
\end{theorem}

\begin{proof}
  We have already defined $\sigma$ and $\sigma^{\geo}$ in
  Theorems~\ref{the:The_periodic_simple_structure_set_of_M}
  and~\ref{the:The_geometric_simple_structure_set_of_M} respectively.

Let $\widehat \rho$ be  the composite 
\begin{multline}
\widehat  \rho   \colon  \strp_{n+l+1}(B\Gamma)
\xrightarrow{\left(\prod_{(P) \in \calp}
    \res_{\Gamma_{\ab}}^{P'}\right) \circ \strp_{n+l+1}(B\!\pr)}
\prod_{(P) \in \calp} \strp_{n+l+1}(BP')
\\
\xleftarrow{\cong} \prod_{(P) \in \calp} \widetilde{L}^s_{n+l+1}(\IZ P')[1/p]
\xrightarrow{\prod_{(P) \in \calp}  \widetilde{\sign}_{P'}} 
\prod_{(P) \in \calp}  \wt R(P')^{(-1)^d}[1/p],
 \end{multline}
where for $(P) \in \calp$ the subgroup $P' \subset \Gamma_{\ab}$ is
the image of $P$ under the projection $\pr \colon \Gamma \to
\Gamma_{\ab}$, the first map
is the isomorphism taken from
Theorem~\ref{the:Computation_of_strp(BGamma)}~%
\ref{the:Computation_of_strp(BGamma):restriction_iso},
the second map is given by product of the inverse of the isomorphism
appearing in 
Theorem~\ref{the:The_periodic_structure_set_of_BP_for_a_finite_p-group}~%
\ref{the:The_periodic_structure_set_of_BP_for_a_finite_p-group:localization},
and the third map is the product of the homomorphisms defined in~\eqref{multi_signature-widetilde}.
Define $\rho$ to be the composite
\[
\rho   \colon  \strp_{n+l+1}(M)
\xrightarrow{\strp_{n+l+1}(f)} \strp_{n+l+1}(B\Gamma) \xrightarrow{\widehat  \rho} \prod_{(P) \in \calp}  \wt R(P')^{(-1)^d}[1/p].
\]
Note that all these maps are isomorphisms after tensoring with $\Z[1/2p]$.

We thus see that $\widehat \rho[1/2]$ is an isomorphism and, thus, since the domain of
$\widehat \rho$ is torsionfree by Theorem~\ref{the:Computation_of_strp(BGamma)}~%
\ref{the:Computation_of_strp(BGamma):induction_iso}, $\widehat \rho$ is injective.

The map $\sigma \times \strp_{n+l+1}(f)$ is injective and an isomorphism after tensoring
with $\Z[1/p]$ by Theorem~\ref{the:The_periodic_simple_structure_set_of_M}, so it follows
that
$\sigma \times \rho =(\id \times \widehat \rho) \circ (\sigma \times \strp_{n+l+1}(f))$ is
 injective and is an isomorphism after tensoring with $\Z[1/2p]$.
 
Define $\rho^{\geo} = \rho \circ j(M)$ and note that $j(M)$ is injective by Lemma~\ref{lem:strg_strp_injective}.
\[
\xymatrix{0 \ar[d]&& 0 \ar[d] \\
\strg(M) \ar[d]^-{j(M)} \ar[rr]^-{\sigma^{\geo} \times \rho^{\geo}}&& \left(\bigoplus_{i = 0}^{n-1}
L_{n-i}(\IZ)^{\binom{n}{i}}\right)^{\Z/p} \oplus \ar[d]^-{\inc \times \id} \left(\bigoplus_{(P) \in \calp}
 \wt R(P')^{(-1)^d}[1/p]\right)\\
 \strp_{n+l+1}(M) \ar[d] \ar[rr]^-{\sigma\times \rho}&& \left(\bigoplus_{i = 0}^{n}
L_{n-i}(\IZ)^{\binom{n}{i}}\right)^{\Z/p} \oplus \ar[d] \left(\bigoplus_{(P) \in \calp}
 \wt R(P')^{(-1)^d}[1/p]\right) \\
 \cok (j(M)) \ar[rr]^-{\overline{\sigma \times \rho}} \ar[d]  && \cok(\inc \times \id) \ar[d] \\
 0 && 0
 }
\]
The map $\overline{\sigma \times \rho}$ is the map induced on cokernels by the commutative
square above it.  The columns are short exact sequences so the snake lemma applies and
yields the following exact sequence:
\begin{multline*}
 0 \to \ker(\sigma^{\geo} \times \rho^{\geo}) \to \ker(\sigma \times \rho) \to \ker(\overline{\sigma \times \rho}) \to \\
   \cok(\sigma^{\geo} \times \rho^{\geo}) \to \cok(\sigma \times \rho) \to \cok(\overline{\sigma \times \rho}) \to 0.
\end{multline*}
Since we have already shown that $\ker(\sigma \times \rho) = 0$ it follows that
$\sigma^{\geo} \times \rho^{\geo}$ is injective.  Since we have already shown that
$(\sigma \times \rho)[1/p]$ is an isomorphism, it follows that both
$\cok(\sigma \times \rho)[1/p]$ and $\cok(\overline{\sigma \times \rho})[1/p]$ vanish.
Note that $\cok(j(M))$ is an infinite cyclic group by Lemma~\ref{lem:strg_strp_injective}
and
Theorem~\ref{the:The_geometric_simple_structure_set_of_M}~\ref{the:The_geometric_simple_structure_set_of_M:partial(M)}.
Since the domain and codomain of $\overline{\sigma \times \rho}$ are infinite cyclic and
the cokernel is $p$-torsion, its kernel is trivial.  Hence there is a short exact sequence
\[
0  \to  \cok(\sigma^{\geo} \times \rho^{\geo}) \to \cok(\sigma \times \rho) \to \cok(\overline{\sigma \times \rho}) \to 0
\]
where the middle term is $p$-torsion.   Hence all the groups are $p$-torsion and $(\sigma^{\geo} \times \rho^{\geo})[1/p]$ is an isomorphism as desired.  
\end{proof}

We next interpret the detecting maps from Theorem~\ref{the:geometric_structure_set_of_M}
in terms of geometric invariants.  The first map $\sigma^{\geo}$ will be given by {\em
  splitting invariants}, coming from surgery obstructions along submanifolds and the
second map $\rho^{\geo}$ will be given by \emph{$\rho$-invariants} which arise both in
index theory for manifolds with boundary as well as in the homeomorphism classification of
homotopy lens spaces with odd order fundamental group.  
%

\subsection{Splitting invariants}
Let $X$ be a closed oriented manifold with dimension $d \geq 5$.  Recall that in
Section~\ref{subsec:structure_sets} we mentioned an identification of the geometric
surgery exact sequence with the 1-connective algebraic surgery exact sequence.  To discuss
splitting invariants we say precisely what this identification is.  In particular, we
discuss the bijection $\caln(X) \to H_d(X; \bfL\langle 1 \rangle)$.

Recall that a degree one normal map $(f,\ol f)$ is a degree one map $f: N^d \to X^d$ and a
trivialization $\ol f \colon TN \oplus f^*\xi \cong \ul\R^l$ for stable some bundle $\xi$ over
$X$. Then $\caln(X)$ is the set of normal bordism classes
of degree one normal maps to $X$.  The map $\eta \colon \strg(X) \to \caln(X)$ is given by
considering a simple homotopy equivalence $h : X' \to X$ as a degree one map and taking
$\xi = h^{-1*}\nu_{X'}$, where $h^{-1}$ is a homotopy inverse of $h$ and $\nu_{X'}$ is the
normal bundle of $X'$ with respect to some embedding in Euclidean space.  The map
$\sigma: \caln(X) \to L_d^s(\Z[\pi_1(X)])$ is given by the surgery obstruction.  In
particular a normal bordism class maps to zero if and only if it has a representative
given by a simple homotopy equivalence.
  
A Pontryagin-Thom construction gives a bijection
$PT \colon \caln(X) \xrightarrow{\cong} [X,G/TOP]$.  Sullivan/Quinn/Ranicki give a 4-fold
periodicity, $\Omega^4 (\Z \times G/TOP) \simeq \Z \times G/TOP$.  The corresponding
spectrum is homotopy equivalent to $\bfL(\Z)$ (which we have been abbreviating $\bfL$).
It follows formally that the 0-space of the 1-connective cover $\bfL\langle 1 \rangle$ is
homotopy equivalent to $G/TOP$ and that $[X,G/TOP] = H^0(X; \bfL\langle 1 \rangle)$.
Furthermore, Ranicki shows that closed, oriented, topological manifolds are oriented with
respect to these two spectra, hence there are Poincar\'e duality isomorphisms
$PD \colon H^i(X; \bfL) \xrightarrow{\cong} H_{d-i}(X;\bfL)$ and
$PD \colon H^i(X; \bfL\langle 1 \rangle) \xrightarrow{\cong} H_{d-i}(X;\bfL\langle 1
\rangle)$.  The following foundational theorem is due to Quinn and Ranicki; references
are~\cite[Theorem 18.5]{Ranicki(1992)} and~\cite[Proposition
14.8]{Kuehl-Macko-Mole(2013)}.

\begin{theorem}\label{thm:geom_is_alg} There is a commutative diagram, with vertical
  bijections,
  \[
    \xymatrix{\cdots \ar[r] & L_{d+1}^s(\Z[\pi_1(X)])  \ar@{=}[d] \ar[r]^-{\partial}
      & \strg(X) \ar[d]_s^{\cong} \ar[r]^-{\eta} &\caln(X) \ar[d]_t^{\cong} \ar[r]^-{\sigma} & L_d^s(\Z[\pi_1(X)])  \ar@{=}[d] \\
      \cdots \ar[r] &L_{d+1}^s(\Z[\pi_1(X)]) \ar[r]^-{\xi^{\langle 1 \rangle}_{d+1}} &
      \strone_{d+1}(X) \ar[r] & H_{d}(X;\bfL\langle 1 \rangle) \ar[r]^-{A_d^{\langle 1
          \rangle}(X)} & L_d^s(\Z[\pi_1(X)]), }
  \]
  where $t$ is the composite
  $\caln(X) \xrightarrow{PT} [X,G/TOP] = H^0(X;\bfL\langle 1 \rangle) \xrightarrow{PD}
  H_{d}(X;\bfL\langle 1 \rangle)$.
\end{theorem}
  
The philosophy of splitting invariants is to detect $\caln(X)$ by submanifolds.  Let $Y^j$
be a closed, oriented submanifold of $X^d$.  (We assume our submanifolds have a normal
bundle so we can apply transversality.)  The restriction map
\[
  \res^Y_X \colon \caln(X^d) \to \caln(Y^j)
\]
  is defined by representing a normal bordism class by a normal map $(f ,\ol f)$ where
  $f \colon N \to X$ is transverse to $Y$.  Then
  $\res [ f,\ol f] $ is represented  by the restricted map $f|\colon f^{-1} Y \to Y$
 with
  the trivialization
  $ Tf^{-1} Y \oplus f|^*(\nu(Y \hookrightarrow X) \oplus \xi|_Y) = TN|_{f^{-1}Y} \oplus
  f|^*\xi \cong \ul \R^l$.  
  
It follows from the Pontragin-Thom construction that the map $ \res^Y_X$ is the composite 
$$\caln(X^d) \xrightarrow{PT} [X,G/TOP] \xrightarrow{[-,G/TOP]}  [Y,G/TOP] \xrightarrow{PT^{-1}}
  \caln(Y),
 $$
 but we state this result slightly differently:
  
\begin{lemma}\label{lem:PD_of_res} $\res^Y_X$ is the composite
\begin{multline*}
  \caln(X^d) \xrightarrow{PT} [X,G/TOP] = H^0(X; \bfL \langle 1 \rangle) \\
  \xrightarrow{\inc^*} H^0(Y; \bfL \langle 1 \rangle) = [Y,G/TOP]  \xrightarrow{PT^{-1}}
  \caln(Y^j).
\end{multline*}
\end{lemma}

\begin{definition}[Sullivan~\cite{Sullivan(1966PhD), Sullivan(1966notes)}] A
  \emph{characteristic variety} of $X^d$ is a collection of closed, oriented,
  connected, submanifolds $\{Y_i^j\}$ of $X^d$ so that any simple homotopy equivalence
  $h : X' \to X$ which vanishes under the compositie
  \[
    \caln(X^d) \xrightarrow{\res} \prod_{i,j} \caln(Y_i^j) \xrightarrow{\sigma}
    \prod_{i,j} L_j(\Z [\pi_1Y_i^j]) \xrightarrow{\varepsilon} \prod_{i,j} L_j(\Z)
  \]
  is homotopic to a homeomorphism.
\end{definition}

Sullivan's original example is that a characteristic variety for $\C P^n$ is given by
$\{\C P^j\}$ with $0< j < n$.  If a manifold satisfies topological
rigidity~\cite{Kreck-Lueck(2009nonasph)} (for example a sphere or a torus), then the empty
set is a characteristic variety.

\begin{theorem}\label{thm:characteristic_variety_for_torus_x_sphere} Choose a point
  $\pt \in S^l$ in the sphere.
  \begin{enumerate}
  \item\label{thm:characteristic_variety_for_torus_x_sphere:(1)} The following map is an
    isomorphism
    $$\inc^* \times (\pr_{T^n*} \circ PD) \colon H^0(T^n \times S^l; \bfL \langle 1
    \rangle) \to H^0(T^n \times \pt; \bfL \langle 1 \rangle) \oplus H_{n+l}(T^n; \bfL
    \langle 1 \rangle).$$
  \item\label{thm:characteristic_variety_for_torus_x_sphere:(2)} For a subset
    $J \subset \{1,2, \dots, n\}$, let $T^J \times \pt \subset T^n \times S^l$ be the
    obvious $|J|$-dimensional submanifold.  Then
    $\{T^J \times \pt \mid \emptyset \not = J \subset \{1,2, \dots, n\}\}$ is a
    characteristic variety for $T^n \times S^l$.
  \item\label{thm:characteristic_variety_for_torus_x_sphere:(3)} Let
    $\sigma^{\geo} \colon \strg(M) \to H_n(T^n;\bfL\langle 1 \rangle)^{\IZ/p}$ be the map
    defined in Theorem~\ref{the:The_geometric_simple_structure_set_of_M}.  Then
    $\sigma^{\geo}(h : N \to M) = 0$ if and only if the $p$-fold cover
    $\ol h : \ol N \to T^n \times S^l$ is homotopic to a homeomorphism.
  \end{enumerate}
\end{theorem}

\begin{proof}~\ref{thm:characteristic_variety_for_torus_x_sphere:(1)} It is an exercise to
  show the analogue in ordinary homology
 $$
 \inc^* \times (\pr_{T^n*} \circ PD) \colon H^*(T^n \times S^l) \xrightarrow{\cong}
 H^*(T^n \times \pt) \oplus H_{n+l-*}(T^n).
 $$ 
 The general case follows since the Atiyah-Hirzebruch spectral sequence collapses and the $\bfL$-spectrum Poincar\'e duality is compatible with classical Poincar\'e duality since the fundamental class
 $[T^n \times S^l]_{\bfL^{\pt}}\in H_{n+l}(T^n \times S^l;\bfL^{\pt})$ maps to the
 fundamental class
 $[T^n \times S^l] \in H_{n+l}(T^n \times S^l) \subset \bigoplus H_{n+l-*}(T^n \times S^l;
 L^*(\Z))$.  (See Section 16, The $L$-theory orientation of topology
 in~\cite{Ranicki(1992)}.)
 \\[1mm]~\ref{thm:characteristic_variety_for_torus_x_sphere:(2)} Let $h : X' \to T^n \times S^l$ be a
 simple homotopy equivalence.  We first claim that $h$ is homotopic to a homeomorphism if
 and only if
 \[
   PT(\eta(h))= 0 \in H^0(T^n \times S^l; \bfL \langle 1 \rangle).
 \]
 Since
 $H_{n+l+1}(T^n \times S^l ; \bfL \langle 1 \rangle) \to H_{n+l+1}(T^n ; \bfL \langle 1
 \rangle)$ is a (split) surjection, and the Farrell-Jones Conjecture for $\Z^n$, which is due to
 Shaneson in this case, implies that
 $A_{n+l+1}^{\langle 1 \rangle}(T^n) :H_{n+l+1}(T^n ; \bfL \langle 1 \rangle) \to
 L_{n+l+1}(\Z[\Z^n])$ is an isomorphism, the composite
  $$
  A_{n+l+1}^{\langle 1 \rangle}(T^n) \circ H_{n+l+1}(\pr_{T^n} ; \bfL \langle 1 \rangle)
  =A_{n+l+1}^{\langle 1 \rangle}(T^n \times S^l)
  $$
  is surjective.  Hence the boundary map $\partial$ in the surgery exact sequence (see
  Theorem~\ref{thm:geom_is_alg}) is the trivial map and thus $\eta$ is injective.  Our first claim follows.

  Our second claim is that the map
$$
\prod_J \varepsilon \circ \sigma \circ \res_{T^n}^{T^J} \circ (PT)^{-1} : H^0(T^n; \bfL
\langle 1 \rangle) \to \bigoplus_J L_{|J|}(\Z)
$$
is an isomorphism where $J$ runs over all nonempty subsets of $\{1,2,\dots, n\}$.  To see
this is an isomorphism we exhibit its inverse map which sends the generator
$E^{|J|} \to S^{|J|}$ of $L_{|J|}(\Z)$ to $PT((T^{J} \# E^{|J|}) \times T^{J'} \to T^n)$
where $J'$ is the complement of $J$.

Our third claim is that for a simple homotopy equivalence $h: X' \to T^n \times S^l$,
\begin{equation}\label{vanish} \pr_{T^n*}(PD(PT(\eta(h)))) = 0 \in
  H_{n+l}(T^n;\bfL\langle 1 \rangle).
\end{equation}
Indeed,
\begin{align*}
  0 & = \sigma(\eta(h)) & \text{by exactness of the surgery exact sequence}  \\
    & = A_{n+l}^{\langle 1 \rangle}(T^n \times S^l)(PD(PT(\eta(h)))) & \text{by Theorem~\ref{thm:geom_is_alg}} \\
    & =  A_{n+l}^{\langle 1 \rangle}(T^n) \pr_{T^n*}(PD(PT(\eta(h)))) & \text{by naturality of the assembly map.}
\end{align*}
Since $A_{n+l}^{\langle 1 \rangle}(T^n)$ is an isomorphism,~\eqref{vanish} follows.

Now back to the proof of~\ref{thm:characteristic_variety_for_torus_x_sphere:(2)}.  Let
$h : X' \to T^n \times S^l$ be a simple homotopy equivalence so that all splitting invariants along $T^J \times \bullet$ vanish.  Our goal is to show that $h$ is
homotopic to a homeomorphism.  By our first claim,
part~\ref{thm:characteristic_variety_for_torus_x_sphere:(1)}, and our third claim, it
suffices to show that
\[
  \inc^*(PT(\eta(h))) = 0 \in H^0(T^n \times \pt; \bfL\langle 1 \rangle).
\]
By Lemma~\ref{lem:PD_of_res}, this is equivalent to showing
\[
  PT(\res_{T^n \times S^l}^{T^n \times \pt}(\eta(h))) = 0 \in H^0(T^n \times \pt;
  \bfL\langle 1 \rangle)
\]
By our second claim, this is equivalent to showing that for any nonempty subset $J$ of
$\{1,2, \dots, , n\}$,
\[
  \varepsilon(\sigma(\res_{T^n \times S^l}^{T^J \times \pt} (\eta(h) ))) = 0 \in
  L_{|J|}(\Z).
\]
It follows that $\{ T^J \times \pt\} $ is a characteristic variety for $T^n \times S^l$.
\\[1mm]~\ref{thm:characteristic_variety_for_torus_x_sphere:(3)} The composite
$\sigma^{\geo} \colon \strg(M) \to H_n(T^n;\bfL\langle 1 \rangle)^{\IZ/p}$ with the
inclusion \\$H_n(T^n;\bfL\langle 1 \rangle)^{\IZ/p} \hookrightarrow H_n(T^n;\bfL\langle 1
\rangle)$ can be identified with the composite
\begin{multline*}
  \strg(M) \xrightarrow{\tr} \strg(T^n \times S^l) \xrightarrow{\eta} \caln(T^n \times
  S^l) \\\xrightarrow{\res} \caln(T^n \times \pt) \xrightarrow{PT} H^0(T^n; \bfL \langle 1
  \rangle) \xrightarrow{PD} H_n(T^n; \bfL \langle 1 \rangle),
\end{multline*}
where the transfer map tr is passing to the
$p$-fold cover.  By
Theorem~\ref{thm:characteristic_variety_for_torus_x_sphere}\ref{thm:characteristic_variety_for_torus_x_sphere:(1)}
and Lemma~\ref{lem:PD_of_res}, the result follows.
\end{proof}
       
\begin{definition}
  A map $f \colon N \to X$ \emph{splits along a submanifold $Y$} if
  $f$ is homotopic to a map $g$, transverse to $Y$ so that $g^{-1}Y \to
  Y$ is a homotopy equivalence.  The map $f$ \emph{s-splits along
    $Y$}, if, in addition, $g^{-1}Y \to Y$ is a simple homotopy equivalence.
\end{definition}

Thus if $\{Y^j_i\}$ is a characteristic variety for $X$, then a simple homotopy
equivalence is homotopic to a homeomorphism if and only if it splits along $\{Y^j_i\}$.

We now mention the relationship between restriction and splitting invariants.
The key result is due to Browder (see~\cite[Section 4.3]{Weinberger(1994)}).

\begin{theorem}\label{thm:Browder_splitting_theorem} Suppose $h : X' \to X$ is a simple
  homotopy equivalence and that $Y^j$ is a submanifold of $X^d$ with codimension
  $d-j\geq 3$ and dimension $j \geq 5$.  Then $h$ is splittable along $Y$ if and only if
  $\sigma(\res_X^Y(h)) = 0 \in L_j^s(\Z[\pi_1Y])$.
\end{theorem}
    
\subsection{Rho invariants}

\begin{definition}\label{Rho-invariant} Let $N$ be a closed, oriented,
  $(2d-1)$-dimensional manifold mapping to $BG$ for a finite group $G$.  The
  \emph{$\rho$-invariant}
  \[
   \rho(N \to BG) \in \wt R(G)^{(-1)^d}[1/|G|]
 \]
 is defined by
  \[
    \rho(N \to BG)=\frac{1}{k} \cdot \sign_G(W) \in \wt R(G)^{(-1)^d}[1/|G|]
  \] where $k$ is a power of $|G|$ and $W$ is a compact, oriented $2d$-dimensional
  manifold with orientation preserving free $G$-action, whose boundary is $k$ disjoint
  copies of the $\ol N$, the induced $G$-cover of $N$.  This was given an analytic
  interpretation by Atiyah-Patodi-Singer.  The \emph{$\rho'$-function}
  \[
    \rho'(N \to BG) \colon \strg(N) \to R(G)^{(-1)^d}[1/|G|]
  \]
  is defined by
  \[
    \rho'(N \to BG)(N' \to N) = \rho(N' \to N \to BG) - \rho(N \to BG).
  \]
\end{definition}
 
The $\rho$-invariant and the $\rho'$-function only depend on the induced homomorphism
$\pi_1 N \to G$, or, equivalently, on the induced $G$-cover $\ol N \to N$.  When
$\pi_1 N = G$, we will write $\rho(N)$ and $\rho'$.

The passage from $R(G)^{(-1)^d}[1/|G|]$ to $ \wt R(G)^{(-1)^d}[1/|G|]$ ensures that the
$\rho$-inva\-riant is independent of the choices of $k$ and $W$.  For the definition of the
$\rho$-invariant see~\cite[Remark after Corollary~7.5]{Atiyah-Singer(1968b)}
or~\cite[Section 13B] {Wall(1999)}.

Here are two fundamental properties of the $\rho'$-function.

\begin{theorem}\label{thm:rho'} Let $N \to BG$ be a map from a closed, oriented,
  $(2d-1)$-dimensional manifold to the classifying space of a finite group.  Let
  $\phi: \pi_1N \to G$ be the induced map on fundamental groups.  Recall there is an
  identification $\strg(N) = \strone_{2d}(N)$ (see~\eqref{ident_structure_sets}).  In
  particular the geometric structure set is an abelian group.
  \begin{enumerate}
  \item\label{thm:rho':(1)} The map $\rho'(N \to BG)\colon \strg(N) \to \wt R(G)^{(-1)^d}[1/|G|]$
    is a homomorphism of abelian groups.
  \item\label{thm:rho':(2)} Recall that $L^s_{2d}(\Z[\pi_1N])$ acts on $\strg(N)$
    (see~\cite[Theorem 10.4]{Wall(1999)}).  For $x \in L^s_{2d}(\Z\pi_1N)$ and
    $y \in \strg(N)$,
    \[
      \rho'(N \to BG)(x + y) = \widetilde{\sign}_{G}(\wt L^s_{2d}(\phi)(x)) + \rho'(N \to
      BG)(y) \in \wt R(G)^{(-1)^d}[1/|G|].
    \]
    In particular, by taking $y = \id_N$, we have an equality of maps
    $$ \rho'(N \to BG) \circ \partial = \wt \sign_G \circ \wt L^s_{2d}(\phi) \colon \wt
    L^s_{2d}(\Z \pi_1 N) \to \wt R(G)^{(-1)^d}[1/|G|].$$
  \end{enumerate}

\end{theorem}

\begin{proof}~\ref{thm:rho':(1)} This is the main result of the
  paper~\cite{Crowley-Macko(2011)} by Crowley and Macko.  \\[1mm]~\ref{thm:rho':(2)} The
  following is an easy consequence of the definition of the $\rho$-invariant: If $W^{2d}$
  is a compact oriented manifold with a map to $BG$ and if
  $(\partial W \to BG) = (N' \sqcup - N'' \to BG)$, then
  $\rho(N' \to BG) = \sign_G(\ol W) + \rho(N'' \to BG)$.

  The action is implemented by such a $W$ and the result follows.
\end{proof}

Let $L^l$ be a homotopy lens space with fundamental group $P \cong \Z/p$, for $p$ an odd
prime.  Let $d = (l+1)/2$.  A homotopy lens space is the orbit space of a free action of
$\Z/p$ on $S^l$; equivalently it is a closed manifold having the homotopy type of a lens
space.

\begin{theorem}
  $\rho': \strg(L^l) \to \wt R(P)^{(-1)^d}[1/p] \cong \Z[1/p]^{(p-1)/2}$ is an injection,
  and is an isomorphism after tensoring with $\Z[1/2p]$.
\end{theorem}

\begin{proof}
  We first show that $\strg(L^l)$ is isomorphic to $\Z^{(p-1)/2}$ and is in particular
  torsion-free.  The Atiyah-Hirzebruch Spectral Sequence in equivariant homology
  (see~\cite[Theorem 4.7]{Davis-Lueck(1998)}) shows that
  $H_{l+1}^P(S^l \to S^\infty; \bfL \langle 1 \rangle)$ is zero.  The long exact sequence
  of the triple $S^l \to S^\infty \to \pt$ then shows that
  $\strone_{l+1}(L^l) \to \strone_{l+1}(BP)$ is injective.  But the domain is
  $\strg(L^l)$, which is finitely generated and has rank $(p-1)/2$ by the surgery exact
  sequence, and the codomain is isomorphic to $\Z[1/p]^{(p-1)/2}$ by
  Theorem~\ref{the:The_periodic_structure_set_of_BP_for_a_finite_p-group}.

  The multisignature map $\wt \sign_G : \wt L_{l+1}^s(\Z P) \to \wt R(P)^{(-1)^d}[1/p]$ is
  a $\Z[1/2p]$-iso\-mor\-phism by Lemma~\ref{lem:multisignature_iso_after_inverting_2}.
  We conclude from by   Theorem~\ref{thm:rho'}~\ref{thm:rho':(2)} that  the composite
  $ \wt L_{l+1}^s(\Z P) \xrightarrow{\partial} \strg(L^l) \xrightarrow{\rho'} \wt
  R(P)^{(-1)^d}[1/p]$ is a $\Z[1/2p]$-isomorphism.   Since $\strg(L^l)$ is torsion-free of rank
  $(p-1)/2$ the result follows.
\end{proof}

\begin{remark}
  This gives a new proof of Wall's result~\cite[Chapter 14E]{Wall(1999)} that the
  structure set of a homotopy lens space is detected by the $\rho'$ invariant.  He showed
  that if $L^l$ is a homotopy lens space with fundamental group $\Z/k$ for $k$ odd, that
  $\rho' \colon \strg(L^l) \to \wt R(\Z/k)^{(-1)^d}[1/k]$ is injective.  This follows for
  $k$ a prime power by our argument above, and, we believe, for all $k$ odd by the
  techniques in our paper.  In fact, Wall proved much more and identified the image of the
  map $\rho'$, which takes much more work.

  Wall's result that structure set of lens spaces is detected by the $\rho'$-invariant and
  the fact that splitting invariants detect the structure set of $T^n \times S^l$ were a
  major impetus for our paper.
\end{remark}

Recall that $\calp$ is the set of conjugacy classes of subgroups of order $p$ of
$\g = \Z^n \rtimes \Z/p$.  Recall $|\calp| = p^k$ where $n=k(p-1)$.  Let
$\pr: \g \to \g_{\ab}\cong (\Z/p)^{k+1}$ be the quotient map.  For each $(P) \in \calp$,
let $P' = \pr(P)$, let $\g_P = \pr^{-1}(\pr(P))$, and note that
$\ker \pr \subsetneq \g_P \subsetneq \g$.  Let $M_P \to M$ be the cover corresponding to
the subgroup $\g_P$.  Consider the transfer map
$$
\res_\g^{\g_P} : \strg(M) \to \strg(M_P),
$$
obtained by sending a simple homotopy equivalence $N \to M$ to the covering simple
homotopy equivalence $N_P \to M_P$.

\begin{theorem}[Detection Theorem]
\label{thm:Detection_Theorem}
An element $h \in \strg(M)$ is the trivial element if and only
if $\sigma^{\geo}(h) = 0$ and $\rho'(M_P \to BP')(\res_\g^{\g_P}(h)) = 0$ holds for all $(P) \in \calp$.
\end{theorem}
\begin{proof}
Obviously it suffices to show
that the restriction of the homomorphism of abelian groups 
\[
\rho^{\geo} \colon \strg(M) \to \prod_{(P) \in \calp}  \wt R(P')^{(-1)^d}[1/p]
\]
appearing in Theorem~\ref{the:geometric_structure_set_of_M} to the kernel of homomorphism
$\sigma^{\geo}$ appearing in Theorem~\ref{the:geometric_structure_set_of_M} sends $h$ to
$(\rho'(M_P \to BP')(\res_\g^{\g_P}(h)))_{ (P) \in \calp}$.  The image of the map
\[
\nu^{\geo} \colon \bigoplus_{(P) \in \calp} \widetilde{L}^s_{n+l+1}(\IZ P) \to 
\widetilde{L}^s_{n+l+1}(\IZ \Gamma)
\xrightarrow{\widetilde{\xi}^{\langle 1 \rangle}_{n+l+1}(M)} \strg(M)
\]
defined in Theorem~\ref{the:The_geometric_simple_structure_set_of_M}~%
\ref{the:The_geometric_simple_structure_set_of_M:kernel} is contained in the kernel of
$\sigma^{\geo}$ and has finite $p$-power index.
The map $\rho^{\geo}$ is a homomorphism of abelian groups by Theorem~\ref{thm:rho'}~\ref{thm:rho':(1)}.
Since the image of $\rho^{\geo}$ is a
$\Z[1/p]$-module, we conclude that it suffices to show that the restriction of the
homomorphism $\rho^{\geo}$ to the image of the map $\nu^{\geo}$ sends $h$ to
$(\rho'(\res_\g^{\g_P}(h)))_{ (P) \in \calp}$. Since $\im(\nu^{\geo})$ is contained in
the image of $\xi_{n+l+1}^{\langle 1 \rangle}(M) = \partial $, it suffices to show that restriction of
the homomorphism $\rho^{\geo}$ to the image of $\xi_{n+l+1}^{\langle 1 \rangle}(M)= \partial$ sends
$h$ to $(\rho'(M_P \to BP')(\res_\g^{\g_P}(h)))_{ (P) \in \calp}$.  This follows from
Theorem~\ref{thm:rho'}~\ref{thm:rho':(2)}.
\end{proof}

Now Theorem~\ref{the:zero_in_structure_set_intro} follows directly from
Theorems~\ref{the:geometric_structure_set_of_M},~\ref{thm:characteristic_variety_for_torus_x_sphere},
and~\ref{thm:Detection_Theorem}.

\begin{example}\label{exa:small_values_for_p_k_and_n}
  Take $p = 3$, $k = 1$, $n = 2$ and the $\IZ/3$-action in $\IZ^2$ given by the
  $(n_1,n_2) \mapsto (-n_2,n_1 - n_2)$.  Then $\IZ^2_{\rho}$ is $\IZ[\exp(2\pi/3)]$,
  we have
  \[
  \strg(M) \cong \IZ^3 \oplus \IZ/2.
  \]
  There is one nontrivial splitting obstruction taking values in $L_2(\IZ) \cong \IZ/2$,
  make the corresponding map transversal to $T^2 \times \pt \subset T^2 \times S^3$.
 
  There are three conjugacy classes of  subgroups of order $p$ in $\Gamma$ and each yields a
  $\rho$-invariant type obstruction.
  \end{example}


\typeout{----  The operation of the group of self-homotopy equivalences on the geometric structure set  ------------}

\typeout{----------------------------------  Open questions  -----------------------------------------}


\section{Appendix: Open questions}

\label{sec:Appendix:Open_questions}

As mentioned earlier, our inspiration, our muse, for this paper is Wall's
classification~\cite[Chapter 14E]{Wall(1999)} of homotopy lens spaces $L^l$ with odd order
fundamental group.  This classification is complete.  It consists of the classification of
homotopy types of homotopy lens spaces, the classification of the simple homotopy types
within a homotopy type, the computation of the geometric structure set, and the
computation of the moduli space ${\calm}(L^L)$ of homeomorphism classes of the closed
manifolds within a simple homotopy type.  In the case of homotopy lens spaces it is a bit
easier to state if one considers the polarized homotopy type, fixing an orientation and an
identification of the fundamental group with $\Z/k$.  Then the polarized homotopy type is
given by the first $k$-invariant, lying in the group
$(\Z/k)^\times \subset \Z/k = H^{l+1}(B(\Z/k))$.  The simple homotopy types of polarized
lens spaces are given by the Reidemeister torsion $\Delta(L^l)$ and Wall determines the
set of possible values which occur.  Fixing a simple homotopy type he showed, as mentioned
earlier, that
$$
\rho' : \strg(L^l) \to \wt R(\Z/k)^{(-1)^{(l+1)/2}}
$$
is injective and he computed the image.

From this, it is not difficult to compute the moduli space.  For any closed manifold $X$,
let $\sAut(X)$ be the group of homotopy classes of simple self-homotopy equivalences.
Then $\sAut(X)$ acts on $\strg(X)$ with orbit space the moduli space.  In the case of a
homotopy lens space, then $\sAut(X)$ can be computed directly since $L^l$ is a skeleton of
$K(\Z/k,1)$, or better yet, one can compute the action of $\sAut(X)$ on the set of
$k$-invariants, Reidemeister torsions, and $\rho'$-invariants.  We omit the details.

The discussion above leads to several questions:

\begin{enumerate}

\item Can one describe the image of the injective map of
  Theorem~\ref{the:geometric_structure_set_of_M}
  \[
    \sigma^{\geo} \times \rho^{\geo} \colon \strg(M) \to H_n(T^n; \bfL\langle 1
    \rangle)^{\Z/p} \oplus \left(\bigoplus_{(P) \in \calp} \wt R(P')^{(-1)^d}[1/p]\right)?
  \]
  A first step is to compute the $\rho$-invariant of $M_P \to BP'$.

\item In the proof of the Detection Theorem~\ref{thm:Detection_Theorem} we show that
  $\rho^{\geo}$ restricted to the kernel of $\sigma^{\geo}$ is given by differences of
  $\rho$-invariants.  Is this also true for $\rho^{\geo}$ itself, in other words, does
  $\rho^{\geo} = \prod_{P \in \calp} \rho'(M_P \to BP')\circ \res_\g^{\g_P}$?

\item What can we say about the moduli space of homeomorphism classes of manifolds
  (simple) homotopy equivalent to $M$.  Is it infinite?  Can we compute action of
  $\sAut(M)$ on the structure set?  How does this group act on the splitting and
  $\rho$-invariants?  When is a self-homotopy equivalence of $M$ homotopic to a
  homeomorphism?  Can we classify all self-homotopy equivalences of $M$, perhaps up to
  finite ambiguity?  Does the subgroup of self-homotopy equivalences of $M$ which are
  homotopic to a homeomorphism have finite index in the group of self-homotopy
  equivalences of $M$?  Is a self-homotopy equivalence of $M$ determined by the induced
  map on the fundamental group up to finite ambiguity? Can we show that within the
  homotopy type of $M$ there are infinitely many mutually different homeomorphism types?
  
\item Is a homotopy equivalence $h : N \to M$ splittable along
  \[
    p^k\{\pt \} \times L^l = (T^n)^{\Z/p} \times L^l \subset T^n \times_{\Z/p} S^l = M
  \]
  if  and only if $h$ is a simple homotopy equivalence?  (The if direction follows from
  Theorem~\ref{thm:Browder_splitting_theorem} and equation~\eqref{equation:LsZ/p}.)
  
\item Is a simple homotopy equivalence $h : N \to M = T^n \times_{\Z/p} S^l$ homotopic to
  a homeomorphism if and only if $\ol h : \ol N \to T^n \times S^l$ is splittable along
  $T^J \times \pt$ for all nonempty $J \subset \{1,2,\dots, n\}$ and if $h \simeq k$ where
  $k| : k^{-1}((T^n)^{\Z/p} \times L^l) \to (T^n)^{\Z/k} \times L^l $ is a homeomorphism.
  (An alternate conjecture is that a simple homotopy equivalence $h : N \to M$ is
  homotopic to a homeomorphism if and only if $h \simeq k$ where
  $k| : k^{-1}((T^n)^{\Z/p} \times L^l) \to (T^n)^{\Z/p} \times L^l $ and
  $\ol k | : \ol k^{-1} (T^J \times \pt) \to T^J \times \pt $ are homeomorphisms for all
  $J$. )

\end{enumerate}




\end{document}